\documentclass[ejs]{imsart}

\RequirePackage{amsthm,amsmath,amsfonts,amssymb,epsfig}

\RequirePackage[round,authoryear]{natbib}
\RequirePackage[colorlinks,citecolor=blue,urlcolor=blue]{hyperref}
\RequirePackage{graphicx}

\startlocaldefs

\theoremstyle{plain}

\newtheorem{theorem}{Theorem}[section]
\newtheorem{prop}{Proposition}[section]

\newtheorem{lemma}[theorem]{Lemma}
\theoremstyle{remark}
\newtheorem{definition}[theorem]{Definition}
\newtheorem*{example}{Example}

\theoremstyle{remark}

\def\mR{\mathbb{R}}
\def\mC{\mathbb{C}}

\def\tr{\mbox{tr}}
\def\var{\mbox{var}}

\def\cov{\mbox{cov}}

\def\diag{\mbox{diag}}

\newcommand{\bSig}{\mbox{\boldmath $\Sigma$}}

\newcommand{\bW}{\mbox{\boldmath $W$}}
\newcommand{\bI}{\bold $I$}
\newcommand{\one}{\boldmath{1}}
\newcommand{\E}{\mathbb{E}}

\newcommand{\trans}{^\top}
\newcommand{\um}{\underline{m}}

\def\T{{\bf T}}
\def\X{{\bf X}}

\def\Y{{\bf Y}}

\def\Z{{\bf Z}}

\def\A{{\bf A}}

\def\B{{\bf B}}
\def\R{{\bf R}}

\def\G{{\bf G}}

\def\bI{{\bf I}}
\def\bS{{\bf S}}
\endlocaldefs

\begin{document}
\begin{frontmatter}
\title{Spectral analysis of spatial-sign covariance matrices for heavy-tailed data with dependence}
\runtitle{spatial-sign covariance matrices for heavy-tailed}

\begin{aug}
\author{\fnms{Hantao}~\snm{Chen}\ead[label=e1]{htchen2000@sjtu.edu.cn}
\orcid{0009-0004-7824-2605}}
\and 
\author{\fnms{Cheng}~\snm{Wang}\ead[label=e2]{chengwang@sjtu.edu.cn}\orcid{0000-0003-4752-6088}}
\address{School of Mathematical Sciences, MOE-LSC, Shanghai Jiao Tong University\printead[presep={,\ }]{e1,e2}}
\end{aug}

\begin{abstract}
  This paper investigates the spectral properties of spatial-sign covariance matrices, a self-normalized version of sample covariance matrices, for data from $\alpha$-regularly varying populations with general covariance structures. By exploiting the elegant properties of self-normalized random variables, we establish the limiting spectral distribution and a central limit theorem for linear spectral statistics. We demonstrate that the Mar{\u{c}}enko-Pastur equation holds under the condition $\alpha \geq 2$, while the central limit theorem for linear spectral statistics is valid for $\alpha>4$, which are shown to be nearly the weakest possible conditions for spatial-sign covariance matrices from heavy-tailed data in the presence of dependence.
\end{abstract}

\begin{keyword}[class=MSC]
\kwd[Primary ]{	62H10 }
\kwd[; secondary ]{	62E17}
\end{keyword}

\begin{keyword}
  \kwd{Heavy-tailed data}
  \kwd{limiting spectral statistics}
   \kwd{linear spectral statistics}
   \kwd{random matrix theory}
  \kwd{self-normalization} 
  \kwd{spatial-sign covariance matrices}
\end{keyword}

\end{frontmatter}

\section{Introduction}
Covariance matrices are widely used in multivariate analysis, including principal component analysis, canonical correlation analysis and discriminant analysis. 
For a $p$-dimensional sample $\X_1,\cdots,\X_n \in \mR^p$, the sample covariance matrix is defined as 
\begin{align*}
    \bS=\frac{1}{n}\sum_{i=1}^n\X_i\X_i\trans.
\end{align*}
For high-dimensional data analysis, the theoretical properties of the sample covariance matrix are crucial for establishing the theoretical foundations of many statistical methods. Under multivariate Gaussian assumption, the famous Wishart distribution describes the distribution of $\bS$. To relax the Gaussian assumption and study the behavior of covariance matrices under more general conditions, we consider the independent components model (ICM), i.e.,
\begin{align*}
    \begin{pmatrix}
    \X_1\trans\\
    \vdots\\
    \X_n\trans
    \end{pmatrix}=\begin{pmatrix}
    X_{11} &\cdots &X_{1p}\\
    \vdots & \cdots & \vdots\\
    X_{n1} &\cdots &X_{np}\\
    \end{pmatrix}=\begin{pmatrix}
        Z_{11} &\cdots &Z_{1p}\\
        \vdots & \cdots & \vdots\\
        Z_{n1} &\cdots &Z_{np}\\
        \end{pmatrix}
    \bSig^{1/2}, 
\end{align*}
where $\{Z_{ij},i,j=1,2,\cdots\}$ are independent and identically distributed (i.i.d.) random variables and $\bSig \in \mR^{p \times p}$ is a positive definite matrix. For simplicity, we assume zero mean for the random variables, e.g., $\E Z_{11}=0$.

For the independent components model, random matrix theory extends the classical setting with fixed $p$ to the large-dimensional regime, i.e., 
\begin{align*}
    n\to\infty,\quad p=p_n\to\infty,\quad p/n=y_n\to y\in(0,\infty).
\end{align*}
The seminal work by \citet{marvcenko1967distribution} studied the empirical spectral distribution (ESD) of $\bS$ which converges weakly to the Mar{\u{c}}enko-Pastur (MP) law when $\E Z^2_{11}<\infty$ and $\bSig$ is an identity matrix.   \citet{yin1986limiting} and \citet{silverstein1995strong} extended MP law to general $\bSig$. Assuming $\E Z^2_{11}<\infty$ (without loss of generality, $\E Z^2_{11}=1$)  and regularly conditions on $\bSig$, the limiting spectral distribution (LSD) of the sample covariance matrix converges weakly to a non-random distribution whose Stieltjes transform $m=m(z)\in\mC^+$ is the unique solution to the Mar{\u{c}}enko-Pastur equation
\begin{align} \label{mp-eq}
    m=\int\frac{1}{t(1-y-yzm)-z}dH(t),
\end{align} 
where $H(\cdot)$ is the LSD of $\bSig$. With the LSD, we can obtain the limits of the linear spectral statistics (LSS), e.g.,
\begin{align*}
    \frac{1}{p} \sum_{i=1}^p f(\lambda_i(\bS))=\int f(x) dF^{\bS}(x),
\end{align*}
where $\lambda_i(\bS),~i=1,\cdots,p$ are eigenvalues of $\bS$ and $F^{\bS}$ is the ESD of $\bS$. For the independent components model, \citet{bai2004clt} firstly obtained the central limit theorem (CLT) for the LSS of the sample covariance matrix under the condition $\E Z^4_{11}=3$ and \citet{pan2008central} extended the CLT to the general $\E Z^4_{11}<\infty$. For a comprehensive overview of random matrix theory in statistics, see \citet{bai2010spectral} for example.

From a statistical perspective, the moment condition $\E Z^2_{11}<\infty$ is necessary for the existence of the population covariance matrix, $\cov(\X)$, and for the consistent estimation of the sample covariance matrix. For heavy-tailed distributions where $\E Z^2_{11}=\infty$, the sample covariance matrix may not be a reliable estimator, and more robust alternatives are needed. In this work, we consider the sample spatial-sign covariance matrix (SCM) which is defined as
\begin{align}
    \B=\frac{1}{n}\sum_{i=1}^n\frac{p \X_i\X_i\trans}{\X_i \trans \X_i}.
\end{align}
The SCM was first proposed by \citet{locantore1999robust} for robust principal component analysis. Since then, it has been extensively developed and applied in robust multivariate analysis. See, for example, \citet{visuri2000sign, magyar2014asymptotic, durre2014spatial, durre2015spatial, feng2016multivariate, li2016simpler, chakraborty2017tests} and references therein.

Intuitively, the SCM has entries bounded, exhibits robustness to outliers, and reflects the correlation structure among variables in a similar way to the sample correlation matrix. We expect the SCM to exhibit similar spectral properties to the sample covariance matrix, particularly with respect to the limiting spectral distribution and eigenvalue behavior, even for heavy-tailed populations. Specifically, we focus on  the $\alpha$-regularly varying distribution, a classical model for heavy-tailed data, where $\alpha>0$ controls the decay rate of the tail probability. Similar to the stable law of CLT for universal variables, random matrices with heavy-tailed entries often exhibit complex behavior and typically require non-standard normalization coefficients. For instance, \citet{ben2008spectrum} investigated the limiting spectral distribution of Wigner matrices with heavy-tailed entries and proposed a novel normalization constant. \citet{belinschi2009spectral} extended this analysis to renormalized covariance matrices with heavy-tailed populations. Furthermore, \citet{benaych2014central} established the central limit theorem for linear spectral statistics of Wigner matrices with heavy-tailed entries. Besides the complex renormalizing constant, which depends on the unknown $\alpha$, the limit results in \citet{ben2008spectrum, belinschi2009spectral, benaych2014central} are far from the standard results in random matrix theory for light-tailed distributions. However, the SCM is a self-normalizing version of the sample covariance matrix. The property $\tr(\B)=p$ implies that the SCM is inherently scaled, even for heavy-tailed distributions. Furthermore, since the population spatial-sign covariance matrix, $\E (\B)$, always exists for any $\alpha$, we expect that traditional results for the sample covariance matrix, such as the Mar\u{c}enko-Pastur law and the central limit theorem for linear spectral statistics, may hold under a more general range of $\alpha$ for the SCM. In this work, we focus on the following two questions:
\begin{itemize}
    \item \textbf{Question 1}: Under what range of $\alpha$, the traditional Mar{\u{c}}enko-Pastur equation \eqref{mp-eq} holds for the LSD of SCM?  
    \item \textbf{Question 2}: Under what range of $\alpha$, the traditional CLT for the LSS holds for SCM?
\end{itemize}

In statistics, another self-normalized covariance matrix is the sample correlation matrix. Denoting $\X_{\cdot i}=\left(X_{1i},\cdots,X_{ni}\right)\trans$, the sample correlation matrix is defined as
\begin{align*}
    \R=\left(\frac{\X_{\cdot1}}{\|\X_{\cdot1}\|_2},\cdots,\frac{\X_{\cdot p}}{\|\X_{\cdot p}\|_2}\right)\trans\left(\frac{\X_{\cdot1}}{\|\X_{\cdot1}\|_2},\cdots,\frac{\X_{\cdot p}}{\|\X_{\cdot p}\|_2}\right).
\end{align*}
When $\bSig=\bI$, the LSD of $\R$ was firstly derived by \citet{jiang2004limiting} and the CLT for LSS was studied by  \citet{gao2017high}. For general $\bSig$, it is referred to \citet{heiny2022large} for LSD and \citet{zheng2019test} for CLT. See also \citet{dornemann2022limiting} and \citet{yin2023central}. For heavy-tailed populations with $\bSig=\bI$, \citet{heiny2022limiting} found a new curious LSD under infinite second moment and \citet{heiny2024log} proved a CLT for log-determinant function under infinite fourth moment. Noting that
\begin{align*}
    \B=\frac{p}{n}\left(\frac{\X_1}{\|\X_1\|_2},\cdots,\frac{\X_n}{\|\X_n\|_2}\right)\left(\frac{\X_1}{\|\X_1\|_2},\cdots,\frac{\X_n}{\|\X_n\|_2}\right)\trans,
\end{align*}
the spatial-sign covariance matrix $\B$ has the same structure as $\R$ by interchanging $p$ and $n$ when $\bSig=\bI$. The related results under heavy-tailed populations such as \citet{heiny2022limiting} and \citet{heiny2024log} are still applicable for $\B$. However, for general $\bSig$, the structure of $\B$ becomes completely different with $\R$.

Recently, \citet{li2022eigenvalues} derived the LSD and CLT for the spatial-sign covariance matrix. Specifically, they established that the Mar{\u{c}}enko-Pastur equation \eqref{mp-eq} holds under the moment condition $\E|Z|^{4+\epsilon}<\infty$ and a CLT for general LSS holds when $\E|Z|^5<\infty$. See also \citet{yang2021testing} which considered the CLT of LSS under the weaker condition $\E|Z|^{4}=3$ or $\bSig=\bI$. The limiting behavior of eigenvectors and the related CLT were studied by  \citet{xu2023eigenvectors} under the moment condition $\E|Z|^{8}<\infty$.  In this work, we investigate the LSD and CLT of the spatial-sign covariance matrix for data from heavy-tailed populations with general covariance matrix $\bSig$. Our approach leverages the elegant properties of self-normalized random variables and explores the concentration of quadratic forms. Regarding Question 1, we demonstrate that the Mar{\u{c}}enko-Pastur equation holds under the condition $\alpha \geq 2$. For Question 2, we show that the CLT result of \citet{li2022eigenvalues} continue to  hold under the condition $\alpha>4$. For general $\bSig$, these conditions are shown to be nearly necessary and the weakest possible for the LSD and CLT, respectively.     

The rest of the paper is organized as follows. In Section 2, we present the concentration results for self-normalized random variables. In Section 3, we derive the LSD and CLT results, respectively.  In Section 4, we conduct numerical experiments to justify our theoretical results. In Appendix, we present detailed proofs of our theoretical results.

\section{Preliminaries}\label{sec2}
\subsection{Heavy-tailed populations}
In this paper, we study heavy-tailed populations which are defined by $\alpha$-regularly varying distributions.
\begin{definition}[$\alpha$-regularly varying distribution]
A random variable $Z$ is said to be $\alpha$-regularly varying if 
\begin{align*}
    P(|Z|>x)\sim x^{-\alpha}l(x),
\end{align*} 
where $l$ is a slowly varying function satisfying 
\begin{align*}
    \frac{l(\lambda x )}{l(x)}\to 1,\quad x\to\infty,\quad\forall\lambda>0.
\end{align*}
\end{definition}
For a regularly varying random variable with tail index $\alpha>0$, we have
\begin{gather*}
   \E |Z|^\beta<\infty,~\mbox{if}~\beta<\alpha,\\
   \E |Z|^\beta=\infty,~\mbox{if}~\beta>\alpha.
\end{gather*}
Some commonly used heavy-tailed populations are listed as follows.
\begin{example}[Student's $t$-distribution]\label{exm:t-distribution}
Student's $t$-distribution has the probability density function given by 
\begin{align*}
    f_\alpha(t)=\frac{\Gamma(\frac{\alpha+1}{2})}{\sqrt{\alpha\pi}\Gamma(\frac{\alpha}{2})}\left(1+\frac{t^2}{\alpha}\right)^{-\frac{\alpha+1}{2}}:=c_\alpha\left(\alpha+t^2\right)^{-\frac{\alpha+1}{2}}\sim c_\alpha t^{-\alpha-1},
\end{align*}
where $\alpha>0$ is the numbers of degrees of freedom. It is not difficult to verify that $t$-distribution is also $\alpha$-regularly varying. 
\end{example}
\begin{example}[Pareto distribution]
The probability density of Pareto distribution $F_{\alpha,\theta}$ with parameters $(\alpha,\theta)$ is
\begin{align*}
    f_{\alpha,\theta}(x)=\frac{\alpha\theta^\alpha}{(x+\theta)^{\alpha+1}},\quad x>0.
\end{align*}
For any given $\theta>0$, $F_{\alpha,\theta}$ is $\alpha$-regularly varying.
\end{example}
\subsection{Self-normalized random variables}
For i.i.d. random variables $Z_1,\ldots,Z_p$ from an $\alpha$-regularly varying population, we consider the self-normalized variables
\begin{align*}
    Y_i=\frac{Z_i}{\sqrt{Z_1^2+\cdots+Z_p^2}}.
\end{align*}
Throughout the paper, we assume $P(Z_1=0)=0$. By the property of self-normalization, we have the trivial result 
\begin{align*}
    \E Y_i^2=\frac{1}{p}.
\end{align*}
Furthermore, by exploring the identity $Y_1^2+\ldots+Y_p^2=1$, one can show
\begin{align*}
    1=p \E Y_1^4+p(p-1)\E Y_1^2Y_2^2. 
\end{align*}
Thus, to study the moments of self-normalized variables $Y_1,\ldots,Y_p$, we consider the moment
 \begin{align*}
    \E Y_1^{k_1}\cdots Y_r^{k_r},
 \end{align*}
 for general integers $k_1,\cdots,k_r>0$ and $r\leq p$.

By the identity of Gamma function
\begin{align*}
    \frac{1}{x^\beta}=\frac{1}{\Gamma(\beta)}\int_0^\infty e^{-tx}t^{\beta-1}dt,
\end{align*}
and Fubini's theorem, we have
\begin{align}
    \E Y_1^{k_1}\cdots Y_r^{k_r}=&\E \left( Z_1^{k_1}\cdots Z_{r}^{k_r}\frac{1}{\Gamma(\frac{k}{2})}\int_0^\infty e^{-s(Z_1^2+\cdots+Z_p^2)}s^{\frac{k}{2}-1}ds\notag \right)\\
    =&\frac{1}{\Gamma(\frac{k}{2})}\int_0^\infty s^{\frac{k}{2}-1}\prod_{i=1}^r\E Z_i^{k_i}e^{-sZ_i^2}\prod_{i=r+1}^p\E e^{-sZ_i^2}ds\notag\\
    =&\frac{1}{\Gamma(\frac{k}{2})}\int_0^\infty s^{\frac{k}{2}-1}\prod_{i=1}^r\E Z_i^{k_i}e^{-sZ_i^2}\varphi^{p-r}(s)ds,
    \label{form:calculating_integral}
\end{align}
where $k=k_1+\cdots+k_r$ and $\varphi(s)=\E e^{-sZ_1^2}$ is the Laplace transformation. Furthermore, for even numbers, 
\begin{align*}
	\E Z_i^{2k}e^{-sZ_i^2}=(-1)^k \varphi^{(k)}(s),
\end{align*}
where $ \varphi^{(k)}(s)$ is the $k$-th order derivative of $\varphi(s)$. Thus, asymptotic results of the general moment $\E Y_1^{2k_1}\cdots Y_r^{2k_r}$ can be obtained by analyzing the Laplace transformation $\varphi$.  We refer to \citet{bingham1989regular} and \citet{shao2013self} for more details on the self-normalized random variables.
\begin{prop}\label{lem:even_number}
    If $Z_1,Z_2,\cdots$ are i.i.d. regularly varying with $\alpha\geq 2$, for fixed positive integer $r>0$, we have
    \begin{align*}
        p^r\cdot\E Y_1^{2k_1}\cdots Y_r^{2k_r}\to\begin{cases}
        0,& \#\{1\leq i\leq r:k_i\geq 2\}\geq 1,\\
        1,& k_1=\cdots=k_r=1.
        \end{cases}
    \end{align*}
\end{prop}
By Proposition \ref{lem:even_number}, we have
\begin{gather*}
	\E Y_1^2=\frac{1}{p}(1+o(1)), \quad	\E Y_1^2 Y_2^2=\frac{1}{p^2}(1+o(1)),  \\p \E Y_1^4=1-p(p-1)\E Y_1^2Y_2^2=o(1).
\end{gather*}
If $Z_i$ is symmetrically distributed, it allows to neglect all expectations of odd powers. Generally, to tackle odd numbers $k_1,k_2,\cdots$, \citet{gine1997student} provided a bound and here we refine their results to a more general case.
\begin{prop}\label{lem:odd_number_2}
    If $Z_1,Z_2,\cdots$ are i.i.d. regularly varying with $\alpha\geq 2$ and $\E Z_1=0$, for fixed positive integer $r>0$, we have
    \begin{align*}
        \E Y_1^{k_1}\cdots Y_r^{k_r}=o(p^{-r}),
    \end{align*}
    where $k=k_1+\cdots+k_r$ is even, and at least one of $k_1,\cdots,k_r$ is odd.
    \end{prop}

To tackle CLT, we need more refined results under a stronger condition $\alpha>4$.
\begin{prop}\label{prop:even_number4}
Assuming  $Z_1,Z_2,\cdots$ are i.i.d. regularly varying with $\alpha>4$ and denoting $I_1=\{i:k_i=1\}$, $I_2=\{i:k_i=2\}$ and $I_3=\{i:k_i\geq 3\}$, we have 
\begin{align*}
    p^{|I_1|+2|I_2|}\cdot\E Y_1^{2k_1}\cdots Y_r^{2k_r} \to (\E Z^2)^{-2|I_2|}(\E Z^4)^{|I_2|},
\end{align*}
when $|I_3|=0$. If $|I_3|\geq 1$, there exists $0<\epsilon<\min(1,\alpha/2-1)$ such that
        \begin{align*}
            p^{|I_1|+2|I_2|+(2+\epsilon)|I_3|}\cdot\E Y_1^{2k_1}\cdots Y_r^{2k_r} \to 0.
        \end{align*}
    \end{prop}

\begin{prop}\label{prop:alpha>4_general_number}
    If $Z_1,Z_2,\cdots$ are i.i.d. regularly varying with $\alpha>4$ and $\E Z_1=0$, for fixed positive integer $r>0$, these exist $0<\epsilon'<\min\{1/2,\alpha/2-1\}$ such that
    \begin{align*}
        \E Y_1^{k_1}\cdots Y_r^{k_r}=o\left(p^{-\frac{3}{2}|J_1|-\frac{1}{2}\sum_{i\in J_2}k_i-(2+\epsilon')|J_3|}\right),
    \end{align*}
    where $J_1=\{j:k_j=1\}$, $J_2=\{j:2\leq k_j\leq4\}$ and $J_3=\{j:k_j\geq 5\}$.
\end{prop}

\subsection{Quadratic forms}
Denoting
\begin{align*}
    \Z=\begin{pmatrix}
        Z_1\\
        \vdots\\
        Z_p
    \end{pmatrix},\quad \Y= \begin{pmatrix}
        Y_1\\
        \vdots\\
        Y_p
    \end{pmatrix}=\frac{1}{\sqrt{Z_1^2+\cdots+Z_p^2}}\begin{pmatrix}
        Z_1\\
        \vdots\\
        Z_p
    \end{pmatrix}=\frac{\Z}{\|\Z\|_2},
\end{align*}
we consider the concentration bounds of the quadratic forms
\begin{align*}
    \Y\trans\A\Y,\quad \Y\trans\B\Y,  
\end{align*}
where $\A, \B$ are non-random $p\times p$ symmetric matrices. 
\begin{prop}
    \label{prop:moments_of_self-normalized_quadratic_forms}
    If $Z_1,Z_2,\cdots$ are i.i.d. regularly varying with $\alpha\geq 2$ and $\E Z_1=0$, we have 
    \begin{align*}
        \E\Y\trans\A\Y=\frac{\tr \A}{p}+o(\frac{1}{p}) \|\A\|,
    \end{align*}
 and 
 \begin{align*}
    \E\left|\Y\trans\A\Y-\frac{1}{p}\tr\A\right|^k=o(1)\|\A\|^k,~k \geq 2.
 \end{align*}   
\end{prop}
The concentration of quadratic forms will play an important role in deriving LSD. For CLT, we need more refined results under a stronger condition $\alpha>4$. 
\begin{prop}\label{prop:alpha>4_moments_of_self-normalized_quadratic_forms}
    If $Z_1,Z_2,\cdots$ are i.i.d. regularly varying with $\alpha>4$ and $\E Z_1=0$, we have 
    \begin{align*}
        \E \Y\trans\A\Y=\frac{\tr\A}{p}+O(\frac{1}{p^2})\|\A\|,
    \end{align*}
 and 
\begin{align*}
   &\E \left( \Y \trans \A \Y-\frac{1}{p}\tr \A\right) \left( \Y \trans \B \Y-\frac{1}{p}\tr \B\right)=&\frac{(\tau-3)}{p^2}\tr(\A\circ\B)+\frac{2}{p^2}\tr(\A\B)\\
   &+\frac{1-\tau}{p^3}\tr(\A)\tr(\B)+o(\frac{1}{p})\|\A\|\|\B\|,
\end{align*}
where $\tau=\E Z_1^4$. For $k \geq 3$, 
    \begin{align*}
        \E\left|\Y\trans\A\Y-\frac{1}{p}\tr\A\right|^k=
            o(p^{-1})\|\A\|^k.
    \end{align*}
\end{prop}

\begin{prop}\label{prop:general_self_normalization}
Assume $Z_1,Z_2,\cdots$ are i.i.d. regularly varying with $\E Z_1=0$ and $\bSig \in \mR^{p \times p}$ is strictly positive definite matrix  with $\tr \bSig=p$.    
   \begin{itemize}
    \item When $\alpha \geq 2$:
    \begin{align*}
        \E \frac{\Y \trans \A  \Y }{\Y \trans \bSig \Y}=\frac{\tr\A}{p}+o(1)\frac{\|\A\|\|\bSig\|}{\lambda_{\min}(\bSig)},
     \end{align*}
     and 
     \begin{align*}
        \E\left|\frac{\Y\trans\A\Y}{\Y\trans\bSig\Y}-\frac{\tr\A}{p}\right|^k=o(1)\frac{\|\A\|^k\|\bSig\|^k}{\lambda_{\min}^{k}(\bSig)},~k \geq 2.
     \end{align*}   

     \item When $\alpha>4$:
     \begin{align*}
        \E\frac{\Y\trans\A\Y}{\Y\trans\bSig\Y}=&\frac{\tr\A}{p}+\frac{\tau_3}{p^3}\left( \tr \A \tr(\bSig \circ \bSig)-p \tr(\A \circ \bSig)  \right)\\
        &+\frac{2}{p^3}\left(\tr \A \tr(\bSig^2)-p \tr(\A \bSig) \right)+o(p^{-1})\frac{\|\A\|\|\bSig\|^3}{\lambda_{\min}(\bSig)}.
    \end{align*}
     and 
     \begin{align*}
        &\E\left(\frac{\Y\trans\A\Y}{\Y\trans\bSig\Y}-\frac{\tr\A}{p}\right)\left(\frac{\Y\trans\A\Y}{\Y\trans\bSig\Y}-\frac{\tr\B}{p}\right)\\
        =&\tau_3\left(\frac{\tr\A\tr\B\tr(\bSig\circ\bSig)}{p^4}+\frac{\tr(\A\circ\B)}{p^2}-\frac{\tr\A\tr(\bSig\circ\B)+\tr\B\tr(\bSig\circ\A)}{p^3}\right)\\
        &+2\left(\frac{\tr\A\tr\B\tr(\bSig^2)}{p^4}+\frac{\tr(\A\B)}{p^2}-\frac{\tr\A\tr(\bSig\B)+\tr\B\tr(\bSig\A)}{p^3}\right)\\
        &+o(p^{-1})\frac{\|\A\|\B\|\|\bSig\|^3}{\lambda_{\min}^2(\bSig)},
     \end{align*}
    where $\tau_3=\E Z_1^4-3$. For $k \geq 2$, we have 
    \begin{align*}
        \E\left|\frac{\Y\trans\A\Y}{\Y\trans\bSig\Y}-\frac{\tr\A}{p}\right|^k    =\begin{cases}
            O(p^{-1})\frac{\|\A\|^k\|\bSig\|^k}{\lambda_{\min}^{k}(\bSig)},&k=2,\\
            o(p^{-1})\frac{\|\A\|^k\|\bSig\|^k}{\lambda_{\min}^{k}(\bSig)},&k\geq 3.
        \end{cases}
    \end{align*}
   \end{itemize} 
    
\end{prop}

\begin{prop}\label{prop:cross_term_self_normalization}
    If $Z_1,Z_2,\cdots$ are i.i.d. regularly varying with $\alpha$ and $\E Z_1=0$, 
   \begin{itemize}
    \item $\alpha \geq 2$: 
    \begin{align*}
        \E \left| \sum_{i \neq j} a_{ij} Y_i Y_j\right|^k=\begin{cases}
            O(p^{-1})\|\A\|^k,&k=2,\\
            o(p^{-1})\|\A\|^k,&k\geq 3.
        \end{cases}
     \end{align*}

     \item $\alpha>4 $:
     \begin{align*}
        \E \left| \sum_{i \neq j} a_{ij} Y_i Y_j\right|^k=\begin{cases}
            O(p^{-\frac{k}{2}})\|\A\|^k,&2\leq k\leq 6,\\
            o(p^{-3})\|\A\|^k,&k\geq 7.
        \end{cases}
     \end{align*}
    
   \end{itemize} 

\end{prop}

\section{Main results}
We consider i.i.d. observations $\X_1,\cdots,\X_n$ following
\begin{align*}
  \X_i=\bSig^{\frac{1}{2}}\Z_i,\quad i=1,\cdots,n,
\end{align*}
where $\bSig$ is a $p\times p$ positive definite matrix and normalized by $\tr(\bSig)=p$, and the entries of $\{Z_{ij},i,j=1,2,\cdots\}$ are i.i.d. from $Z$. We assume the following regularity conditions.
\begin{itemize}
    \item [(a)] $n\to\infty$, $p\to\infty$ and $p/n=y_n\to y\in(0,\infty)$.
    \item [(b)] There exists a constant $c>1$, such that $1/c \leq\lambda_{\min}(\bSig)\leq\lambda_{\max}(\bSig)\leq c$.
    \item [(c)] $Z$ is centered and regularly varying with index $\alpha>0$. 
\end{itemize}
We consider the sample spatial-sign covariance matrix 
\begin{align*}
    \B=\frac{p}{n}\sum_{i=1}^{n}\frac{\X_i \X_i\trans}{\|\X_i\|^2_2}=\frac{1}{n}  \sum_{i=1}^{n} \frac{p \bSig^{1/2} \Z_i \Z_i \trans \bSig^{1/2}}{\Z_i \trans \bSig \Z_i}.
\end{align*}
\subsection{Limiting spectral distribution}\label{sec3}
Invoking 
\begin{align*}
    \B=\frac{1}{n}\left(\frac{\sqrt{p}\X_1}{\|\X_1\|_2},\cdots,\frac{\sqrt{p}\X_n}{\|\X_n\|_2}\right)\left(\frac{\sqrt{p}\X_1}{\|\X_1\|_2},\cdots,\frac{\sqrt{p}\X_n}{\|\X_n\|_2}\right)\trans,
\end{align*}
where the columns $\sqrt{p}\X_i/ \|\X_i\|_2,~i=1,\cdots,n$ are independent, \citet{bai2008large} established a general framework to study the LSD of such matrices.  To apply Theorem 1.1 of \citet{bai2008large}, firstly we need to calculate the population covariance matrix of the self-normalized vectors
\begin{align*}
    \T= \E \frac{p \X_1 \X_1\trans}{\|\X_1\|^2_2}=\E \frac{p \bSig^{1/2} \Z_1 \Z_1 \trans \bSig^{1/2}}{\Z_1 \trans \bSig \Z_1}
\end{align*}
and secondly we need to verify the concentrated bound
\begin{align*}
    \sup_{\|\A\|=1}\var\left(\frac{\X_1\trans\A\X_1}{\X_1 \trans \X_1}\right)=o(1),
\end{align*}
which has been studied in Proposition \ref{prop:general_self_normalization}. We next consider the population covariance matrix $\T$. Intuitively,    
\begin{align*}
    \Y_1 \trans \bSig \Y_1 \approx \frac{1}{p} \tr(\bSig)=1,~p \Y_1 \Y_1 \trans \approx \bI,
\end{align*}
which leads to 
\begin{align*}
    \T=\bSig^{1/2} \left(  \E \frac{p  \Y_1 \Y_1 \trans }{\Y_1 \trans \bSig \Y_1}  \right) \bSig^{1/2}\approx \bSig.
\end{align*}
For the special case $\bSig=\bI$,  we can explicitly get the bound 
\begin{align*}
    \|\T-\bI\|_2=&p(p-1)|\E Y_1 Y_2|,\\
\|\T-\bI\|_F^2=&p^3(p-1)|\E Y_1 Y_2|^2,
\end{align*}
which has been studied in Proposition \ref{lem:odd_number_2}. For general $\bSig$, we need to refine the previous results and bound the diagonal and off-diagonal elements separately. We present the main result in the following theorem. 
\begin{theorem} \label{thm:lsd1}
Under Assumptions (b) and (c), if $\alpha \geq 2$ and $p \to \infty$,
\begin{align*}
   \frac{1}{p} \left\|p\cdot\E\frac{\Y\Y\trans}{\Y\trans\bSig\Y}-\bI_p\right\|_F^2 \to 0.
\end{align*}
\end{theorem}

With Theorem \ref{thm:lsd1}, the population covariance matrix $\T$ has the same LSD as the one of $\bSig$. By the main theorem of \citet{bai2008large}, we can derive the LSD of the spatial-sign covariance matrix $\B$.
\begin{prop}\label{thm:LSD_sigma_2}
    Under Assumptions (a)-(c) with $\alpha\geq 2$, assuming the ESD of $\bSig$ converges weakly to $H$, the Mar{\u{c}}enko-Pastur equation \eqref{mp-eq} holds for the LSD of $\B$.
\end{prop}
Proposition \ref{thm:LSD_sigma_2} extends Theorem 2.1 of \citet{li2022eigenvalues} where $\E |Z|^{4+\epsilon}<\infty$ is necessary. Our conclusion is even applicable to the case $\alpha=2$ such as Student's t-distribution with degrees of freedom 2 for which $\E  Z^2=\infty$. 

Interestingly, $\alpha \geq 2$ is the sufficient and necessary condition for classical CLT, i.e., there are constants $a_n,b_n$ such that 
\begin{align*}
    \frac{\sum_{i=1}^p Z_i -a_n}{b_n} \to N(0,1),
\end{align*}
in distribution as $p \to \infty$. We usually call $Z$ with $\alpha \geq 2$ is in the domain of attraction of the normal distribution. 

For $\alpha<2$, \citet{heiny2022limiting} found a curious LSD for the sample correlation matrix $\R$ when $\bSig=\bI$. Due to the connection between $\R$ and $\B$, we can conclude that the conclusion of Proposition \ref{thm:LSD_sigma_2} fails for $\alpha<2$ even $\bSig=\bI$. In other words, $\alpha \geq 2$ is the sufficient and necessary condition for $F^{\B}$ converging to MP law.   

\subsection{Central limit theorem}\label{sec4}
The LSD guarantees the limits of the linear spectral statistics, e.g.,
\begin{align*}
\int f(x) dF^{\B}(x) \to \int f(x) d F^{y,H},
\end{align*}
where $F^{y,H}$ is the distribution function defined by the Mar{\u{c}}enko-Pastur equation \eqref{mp-eq}.  Further, we study the asymptotic distribution of LSSs. Following the proof strategy of \citet{bai2004clt}, firstly we need to calculate $\T=\E(\B)$, and secondly we need to compute the covariance of quadratic forms
\begin{align}
    \E\left(\frac{\X_1\trans\A\X_1}{\X_1\trans\X_1}-\frac{1}{p}\tr\A\right)\left(\frac{\X_1\trans\B\X_1}{\X_1\trans\X_1}-\frac{1}{p}\tr\B\right).\label{form:CLT_cov_quadratic_form}
\end{align}
Both two steps are quite similar to that in proving LSD, but we need to find more precise results here under a stronger condition. For \eqref{form:CLT_cov_quadratic_form}, we derive the results in Proposition \ref{prop:general_self_normalization}. Next, we consider $\T$. By Proposition \ref{prop:general_self_normalization},
\begin{align*}
    &p\E\frac{Y_i^2}{\Y\trans\bSig\Y}\approx 1-\frac{\tau-1}{p}\sigma_{ii}+\frac{2\tr(\bSig^2)+(\tau-3)\tr(\bSig\circ\bSig)}{p^2},\\
    &p\E\frac{Y_iY_j}{\Y\trans\bSig\Y}\approx -\frac{2}{p}\sigma_{ij},
\end{align*}
which leads to 
\begin{align*}
    \T\approx\widetilde{\bSig}=\bSig-\frac{2}{p}\bSig^2-\frac{\tau-3}{p}\bSig^{\frac{1}{2}}\diag(\bSig)\bSig^{\frac{1}{2}}+\frac{2\tr(\bSig^2)+(\tau-3)\tr(\bSig\circ\bSig)}{p^2}\bSig.
\end{align*}
Comparing with $\bSig$, $\widetilde{\bSig}$ has a better approximation error for $\T=\E(\B)$.  
\begin{theorem}\label{thm:clt1}
    Under Assumptions (b) and (c), if $\alpha>4$ and $p\to\infty$,
    \begin{align*}
        p\left\|\T-\widetilde{\bSig}\right\|_F^2\to 0.
    \end{align*}
\end{theorem}
In \citet{li2022eigenvalues}, they obtained the bound $\|\T-\widetilde{\bSig}\|=o(p^{-1})$  under the condition $\E|Z|^5<\infty$, while we get the conclusion under $\alpha>4$.

With Theorem \ref{thm:clt1}, the LSS can be centered as 
\begin{align*}
    G_n(f)=p\left(\int f(x)dF^{\B}(x)-\int f(x)F^{y_n,\widetilde{H}_p}(x)\right),
\end{align*}
where $\widetilde{H}_p$ is the ESD of $\widetilde{\bSig}$ and $F^{y_n,\widetilde{H}_p}$ is defined by the Stieltjes transform $m(z)$ satisfying
\begin{align*}
    m(z)=\int\frac{1}{t(1-y-yzm(z))-z} d\widetilde{H}_p(t),\quad z\in\mC^+.
\end{align*}
And we are ready to derive the CLT for LSS of $\B$.
\begin{prop}\label{thm:clt}
    Under assumptions (a)-(c) with $\alpha>4$, $\E Z^2=1$, $\E Z^4=\tau$, and let $f_1,\cdots,f_k$ be functions analytic on an open interval including 
    \begin{align} \label{support}
        \left[\liminf_n\lambda_{\min}^{\B}\one(0<y<1)(1-\sqrt{y})^2,\limsup_{n}\lambda_{\max}^{\B}(1+\sqrt{y})^2\right],
    \end{align}
    then we have that $\G_n:=(G_n(f_1),\cdots,G_n(f_k))$ converges weakly to a Gaussian random vector $\boldmath{\xi}=\left(\xi_1,\cdots,\xi_k\right)$ with the mean function
    \begin{align*}
        \E (\xi_j)=-\frac{1}{2 \pi i} \oint_{\mathcal{C}_1} f_j(z) \left( \mu_1(z)+(\tau-3)\mu_2(z)\right) dz
    \end{align*}
    and the covariance function
    \begin{align*}
        \cov(\xi_j,\xi_k)=-\frac{1}{4\pi^2} \oint_{\mathcal{C}_1}\oint_{\mathcal{C}_2}f_j(z_1)f_k(z_2)\left(\sigma_1(z_1,z_2)+(\tau-3) \sigma_2(z_1,z_2) \right)dz_1dz_2.
    \end{align*}
where the contours $\mathcal{C}_1$ and $\mathcal{C}_2$ are non-overlapping, closed, counter-clockwise orientated in the complex plane and enclosing the interval \eqref{support}.  
\end{prop}
The mean and covariance functions $\mu_1(z),\mu_2(z),\sigma_1(z_1,z_2),\sigma_2(z_1,z_2)$ can be found in Theorem 2 of \citet{li2022eigenvalues}. Specifically
\begin{gather*}
    \frac{1}{p} \tr \left( \bW_1(z)   \circ \bW_2(z)  \right) \to \mu_2(z).\\
\frac{\partial^2}{\partial z_1 \partial z_2}\left( y \um(z_1) \um(z_2) \frac{1}{p} \tr \left\{ \bW_1(z_1)   \circ \bW_1(z_2)   \right\}  \right)\to \sigma_2(z_1,z_2),
\end{gather*} 
where 
\begin{align*}
    \bW_k(z)=\bSig^{1/2} \left(\bSig+\um(z)\bI\right)^{-k}\bSig^{1/2}-\frac{1}{p}\tr\left(\bSig+\um(z)\bI\right)^{-k} \cdot \bSig,~k=1,2.
\end{align*}
For the special case $\tau=3$, the asymptotic mean $\mu_2(z)$ and variance $\sigma_2(z_1,z_2)$ can be eliminated. \citet{yang2021testing} studied this case. For general $\tau$, they produce effects on the final CLT. 

A degenerate case arise when $\bSig=\bI$ where $\mu_2(z)=0$ and $\sigma_2(z_1,z_2)=0$. In this special case, it is possible to obtain CLT for $\alpha \leq 4$.  For instance, \citet{heiny2024log} obtained the CLT for the log-determinant function of the sample correlation matrix $\R$ when $\bSig=\bI$ and $\tau=\infty$. \citet{li2024necessary} further derived the necessary and sufficient condition for the CLT of general linear spectral statistics under the condition $\bSig=\bI$.  For general $\bSig$, we can conclude $\E  Z^4=\tau<\infty$ is a necessary condition, implying that $\alpha>4$ is nearly weakest possible condition for establishing the CLT for general linear spectral statistics.

\section{Numerical results}\label{sec5}
In this section, we conduct numerical simulations to verify our results. 
\subsection{Empirical performance of spectral distribution}\label{sec4.1}
To investigate the LSD of $\B$, we consider the model
\begin{align*}
    \bSig=\diag(\underbrace{1.2, \cdots, 1.2}_{p/2}, \underbrace{0.8, \cdots, 0.8}_{p/2}),  
\end{align*}
where $Z_{ij},1\leq i\leq n, 1\leq j\leq p$ are i.i.d. random variables following a Student's t-distribution with degrees of freedom $\alpha$.
For this diagonal matrix $\bSig$, the Stieltjes transform $m=m(z)$ is the unique solution to the Mar{\u{c}}enko-Pastur equation
\begin{align*}
  2m=\frac{1}{1.2(1-y-yzm)-z}+\frac{1}{0.8(1-y-yzm)-z}.
\end{align*}
Then we simulate the ESDs of $\B$ with different choices of $\alpha$. 
Figure \ref{Fig1} displays the histograms of empirical eigenvalues for $\alpha \in \{0.5,1,2,4\}$ based on 1000 replications and the theoretical LSDs. We can see that when the degrees of freedom $\alpha\geq 2$, the ESDs fit the theoretical LSDs very well, and when $\alpha<2$, we observe discrepancies between the ESDs and the theoretical LSDs.
\begin{figure}[htbp] 
    \centering 
    \begin{tabular}{cc}				
        \psfig{figure=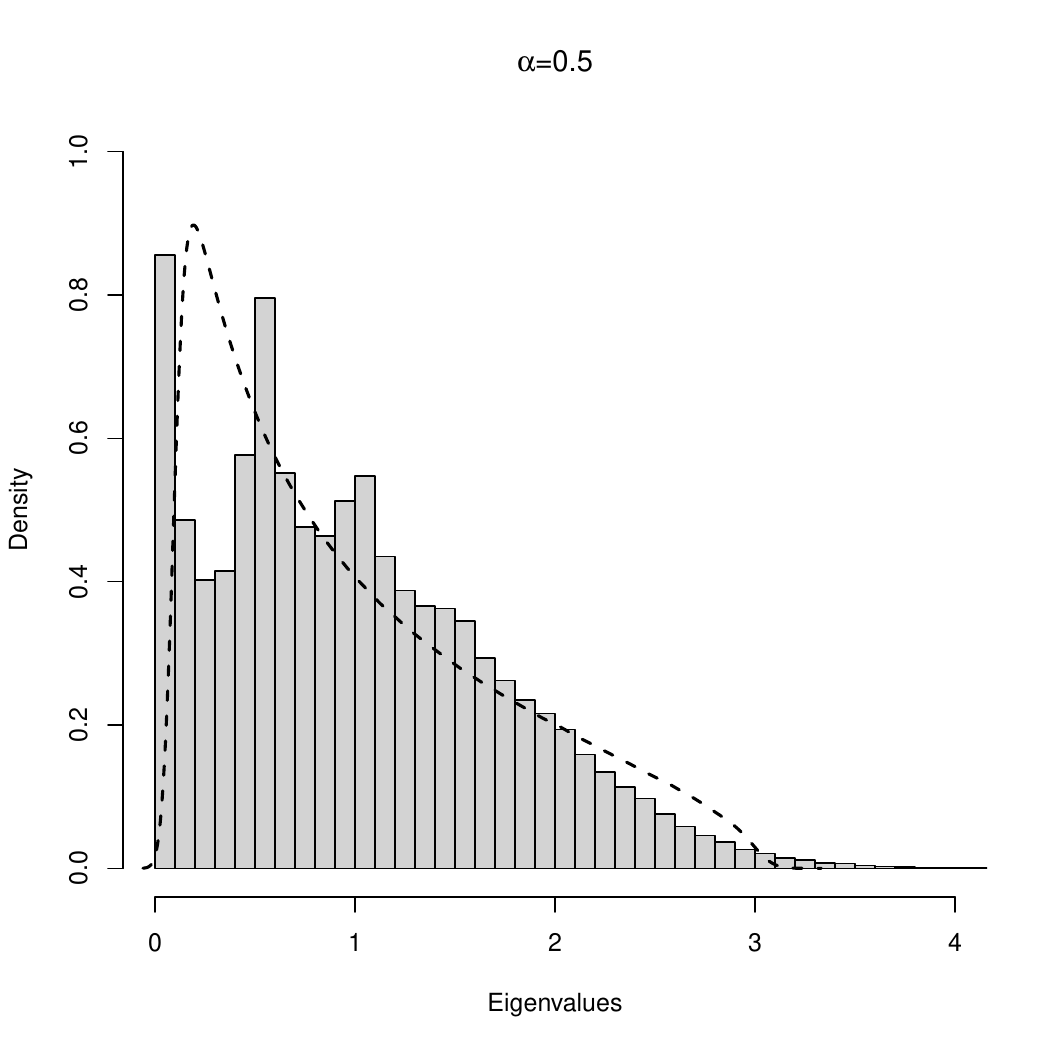,width=2.2 in,angle=0} &
        \psfig{figure=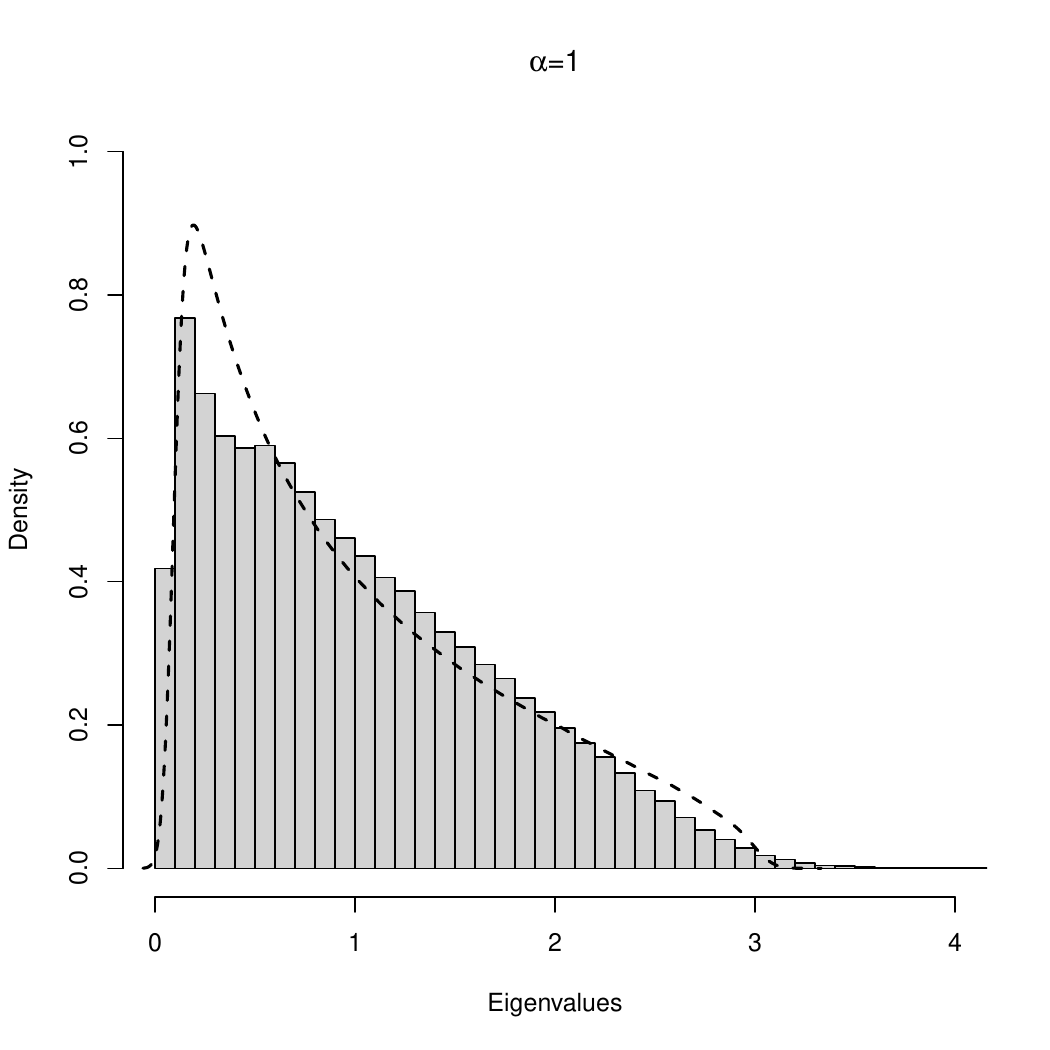,width=2.2 in,angle=0} \\
        \psfig{figure=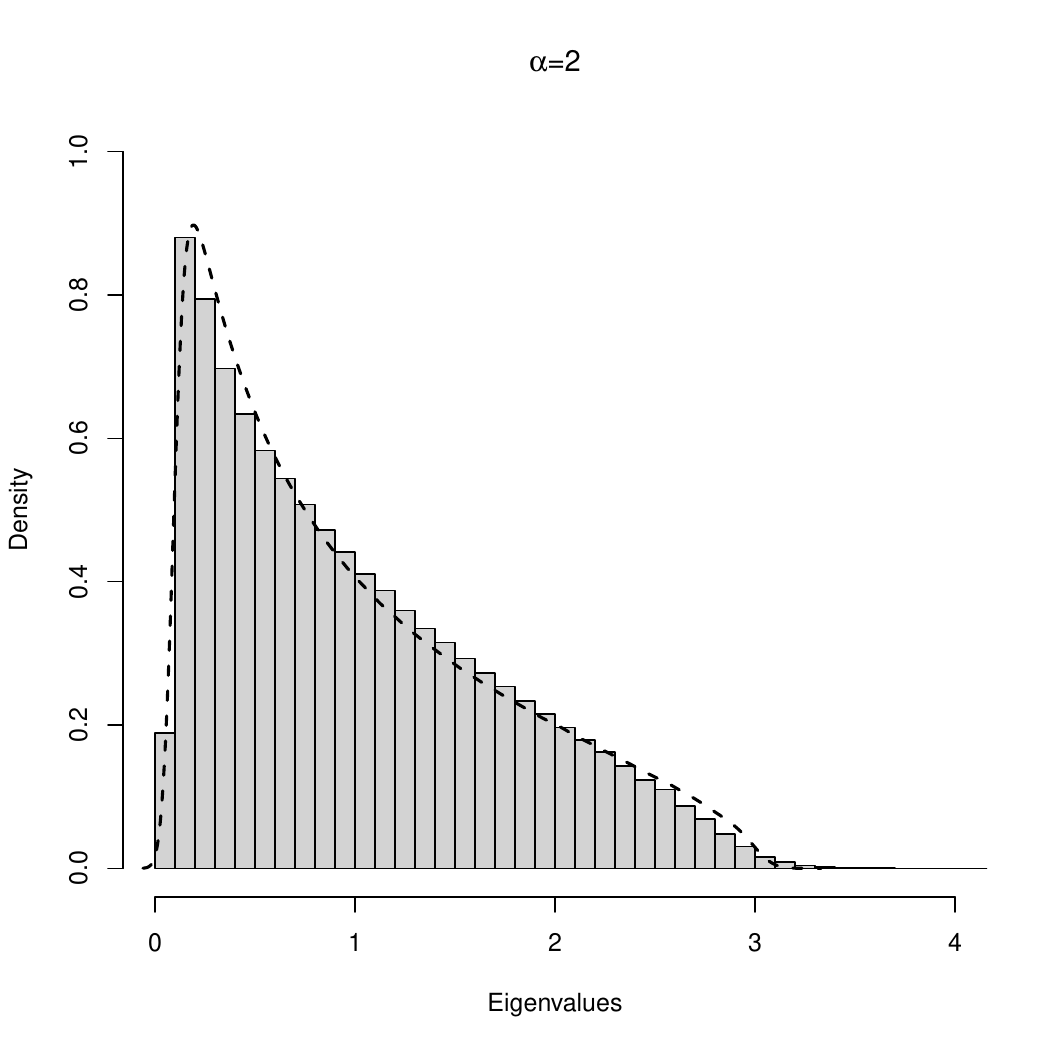,width=2.2 in,angle=0} &
        \psfig{figure=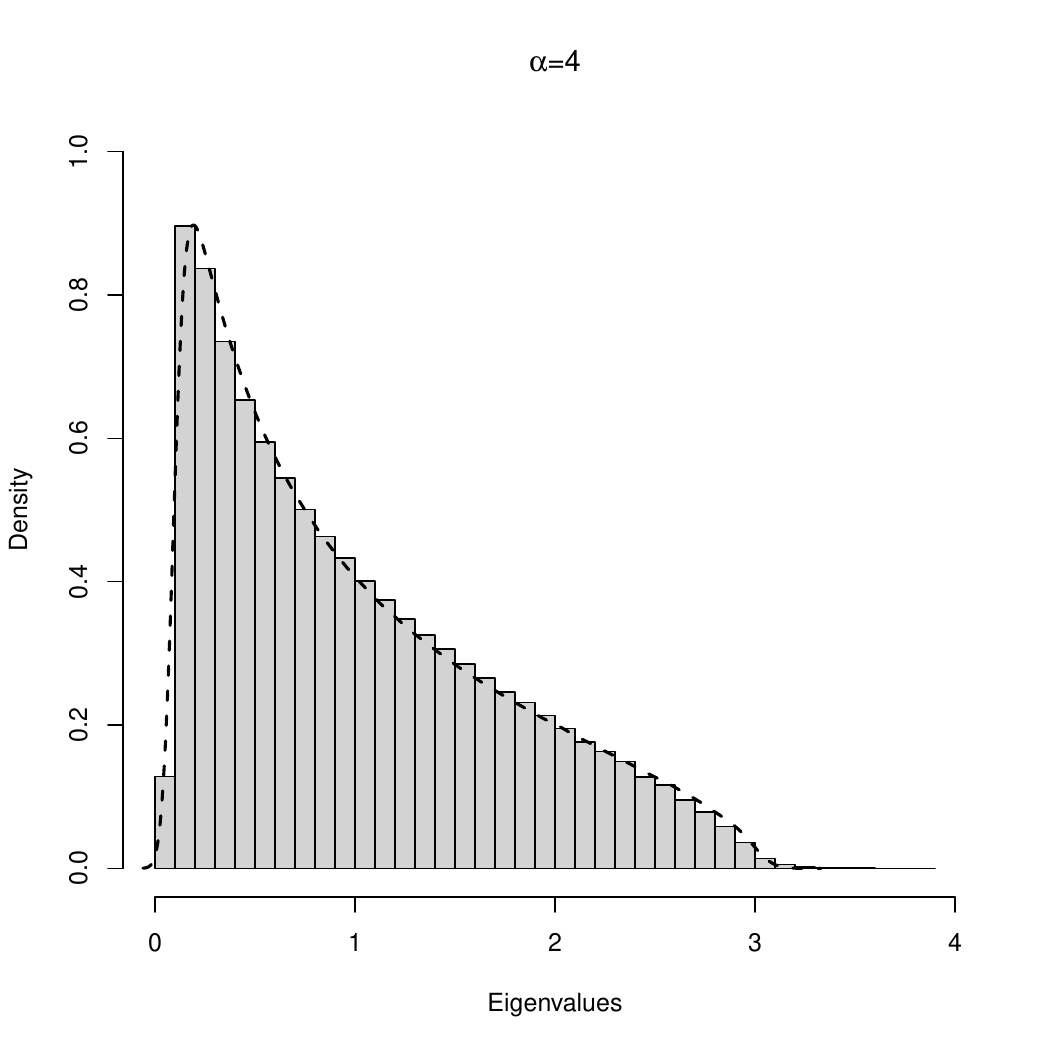,width=2.2 in,angle=0} \\
    \end{tabular}
    \caption{The ESDs of $\B$ with different values of varying index $\alpha\in\{0.5,1,2,4\}$ where $(n,p)=(400,200)$ and the Mar{\u{c}}enko-Pastur law with $y=0.5$.} 
    \label{Fig1} 
  \end{figure}

More precisely, we evaluate the difference between the empirical spectral distribution and the Mar{\u{c}}enko-Pastur law by the infinity norm, i.e., $\|F^{\B}-F^{y,H}\|_{\infty}$. In Figure \ref{Fig2}, we set $p\in\{200,800,2000,5000\}$ with $y=0.5$, and the degrees of freedom $\alpha\in\{0.5,1.0,1.5,2.0,2.5,3.0\}$. 
\begin{figure}[h] 
  \centering 
  \includegraphics[width=1\textwidth]{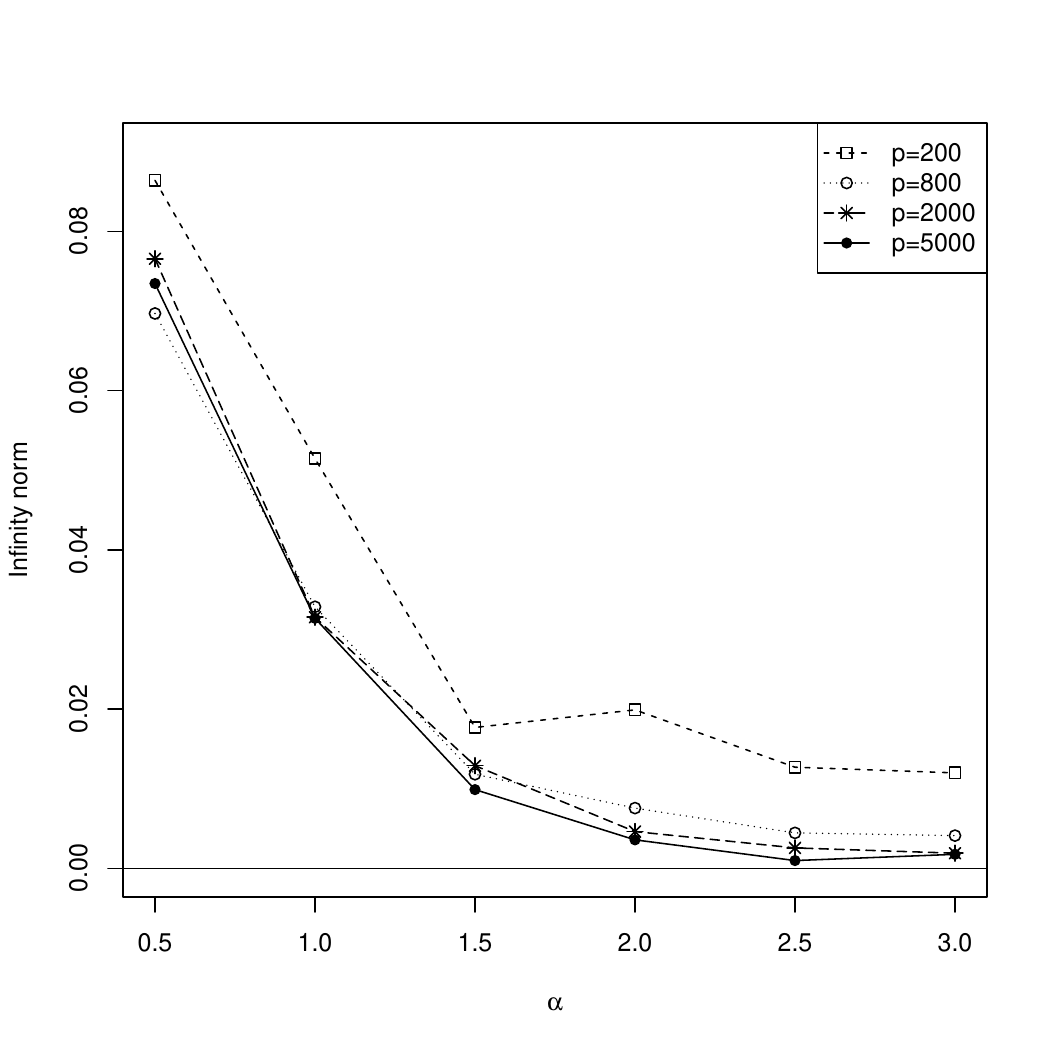} 
  \caption{The infinite norm between ESDs and the generalized Mar{\u{c}}enko-Pastur law  with the growth of $\alpha$ in different settings of dimensions $(n,p)$.} 
  \label{Fig2} 
\end{figure}

On the one hand, if $\alpha<2$, the error terms converge to a positive level when $p$ and $n$ tending to infinity, and it increases dramatically as $\alpha$ decreases to zero. On the other hand, if $\alpha\geq2$, the infinite norm between ESDs and the generalized MP law converge to zero with the growth of dimensions, which means our theoretical conclusions fit the empirical results well.

\subsection{Finite sample performance of CLT}
For the diagonal matrix $\bSig$, \citet{li2022eigenvalues} has calculated the asymptotic mean and variance for $\tr(\B^2)$, that is 
\begin{align*}
    &\mu=-y\alpha_2,\\
    &\sigma^2=4y(\tau-1)\left(\alpha_2^3-2\alpha_2\alpha_3+\alpha_4\right)+4y^2\alpha_2^2,
\end{align*}
where $\alpha_k=\int t^kdH(t)$. And the CLT is given by 
\begin{align*}
    \tr(\B^2)-\tr(\widetilde{\bSig}^2)-\frac{p^2}{n}\to(\mu,\sigma^2).
\end{align*}
where 
\begin{align*}
    \widetilde{\bSig}=\bSig+\frac{\tau-1}{p}\bSig^{1/2} \left( \frac{1}{p} \tr(\bSig^2) \bI-  \bSig \right) \bSig^{1/2}.
\end{align*}
The CLT does not involve $\tau=\E Z^4$ only when $\bSig=\bI$. Otherwise, the bounded fourth moment $\tau<\infty$ is necessary for the CLT.   

In our simulation, we consider $\alpha=4.5$. For 
\begin{align*}
    \bSig=\diag(\underbrace{1.2, \cdots, 1.2}_{p/2}, \underbrace{0.8, \cdots, 0.8}_{p/2}),  
\end{align*}
we have
\begin{align*}
    \tr(\B^2)-1.04p+1.08+\frac{7.22}{p}\to N(-1.04, 6.39).
\end{align*}
Figure \ref{Fig3} present the empirical histograms and the theoretical CLT from which we can see the CLT is still valid for $\E Z^4<\infty$ and $\E |Z|^5=\infty$.

\begin{figure}[h] 
    \centering 
    \begin{tabular}{ccc}				
        \psfig{figure=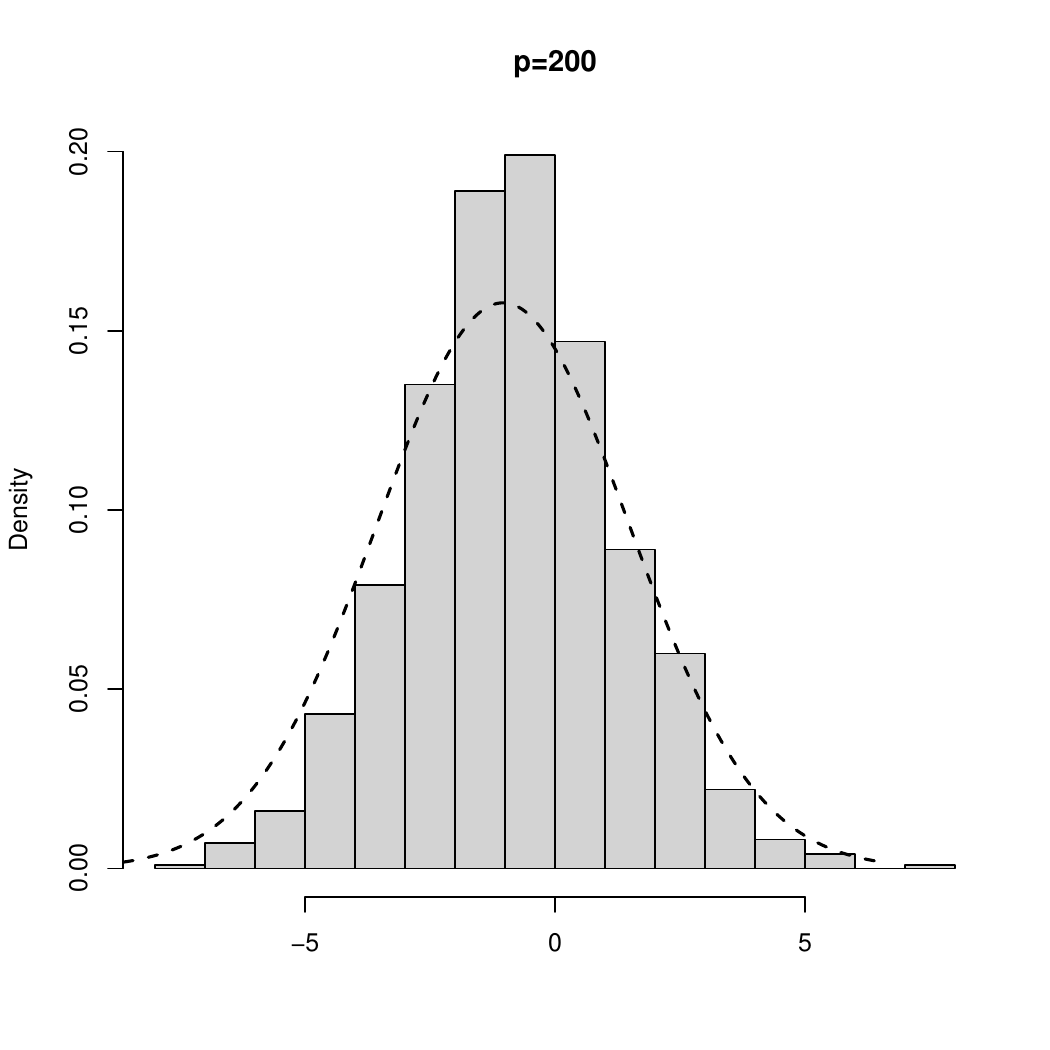,width=1.45 in,angle=0} &
        \psfig{figure=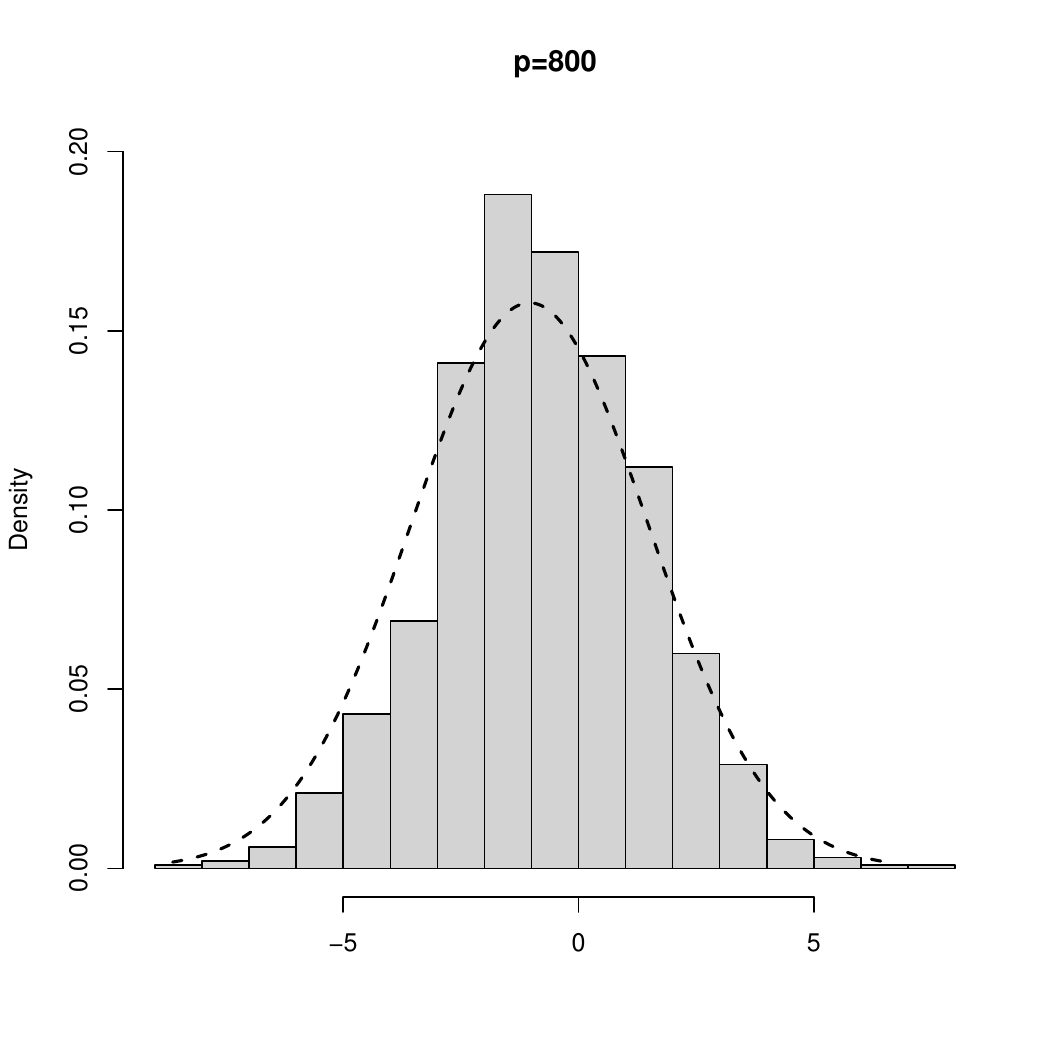,width=1.45 in,angle=0} &
        \psfig{figure=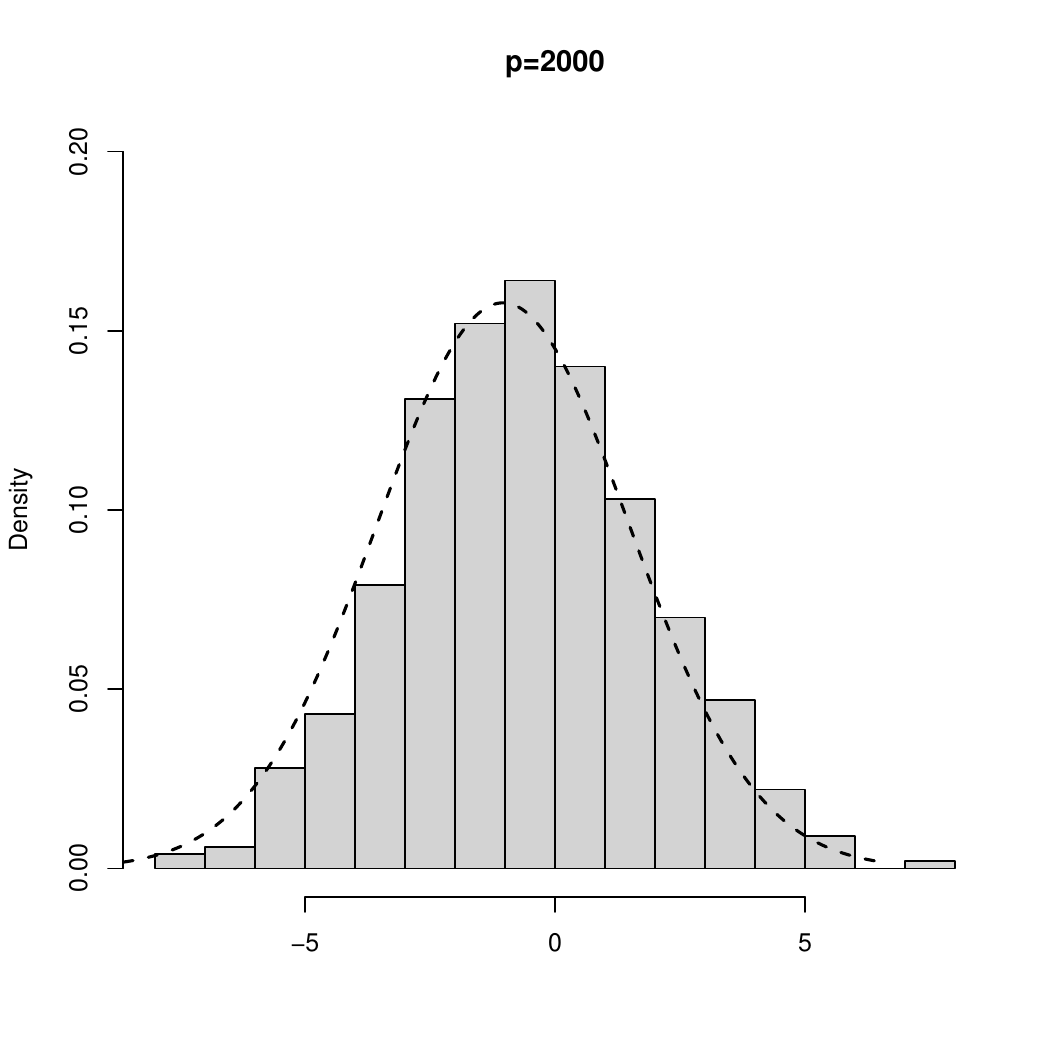,width=1.45 in,angle=0} \\
    \end{tabular}
    \caption{The histograms of statistics $\tr(\B^2)$ and the theoretical CLT for $\alpha=4.5$ and $n=p$.} 
    \label{Fig3} 
\end{figure}


\begin{appendix}
    \section{Auxiliary lemmas}\label{appA}
    \begin{lemma}[Page 7 of \citealt{albrecher2007asymptotic}]\label{lem:exact_limiting}
        Suppose $Z$ is regularly varying with $\alpha>0$. Let $\varphi$ be the Laplace transformation of $Z^2$. Then
        \begin{align*}
            (-1)^k\varphi^{(k)}(s)\sim\begin{cases}
                \frac{\alpha}{2}\Gamma(k-\frac{\alpha}{2})s^{\frac{\alpha}{2}-k}l(\frac{1}{s}),&2k>\alpha,\\
                \frac{\alpha}{2}\widetilde{l}(\frac{1}{s}),&2k=\alpha,\E Z^{2k}=\infty,\\
                \E Z^{2k},&2k\leq\alpha,\E Z^{2k}<\infty,
            \end{cases}\quad(s\to 0^+)
        \end{align*}
        where $\widetilde{l}(x)=\int_{\cdot}^x\frac{l(t)}{t}dt$.
    \end{lemma}
    \begin{lemma}[Corollary 8.1.7 and Theorem 8.8.1 of \citealt{bingham1989regular}]\label{lem:exact_limiting_1-phi}
        Suppose $Z$ is regularly varying with $\alpha>0$. Let $\varphi$ be the Laplace transformation of $Z^2$. Then
        \begin{align}
            1-\varphi(s)\sim\begin{cases}
                \Gamma(1-\frac{\alpha}{2})s^{\frac{\alpha}{2}}l(\frac{1}{s}),&\alpha<2,\\
                s\widetilde{l}(\frac{1}{s}),&\alpha=2,\E Z^{2}=\infty,\\
                \E Z^{2}s,&\alpha\geq2,\E Z^{2}<\infty,
            \end{cases}\quad(s\to 0^+)
        \end{align}
        where $\widetilde{l}(x)=\int_{\cdot}^x\frac{l(t)}{t}dt$.
    \end{lemma}
    \begin{lemma}[Page 349 and 373 of \citealt{bingham1989regular}]\label{lem:norming_constants}
        Let $F$ be a distribution with positive support, and $\varphi$ be its Laplace transformation. For any $\beta \in(0,1]$, if 
        \begin{align*}
            1-\varphi(s)\sim s^{\beta}l(\frac{1}{s}),\quad s\to 0^+,
        \end{align*}
        then, there exists a positive sequence $(a_n)_{n=1}^\infty$, $a_n\to\infty$, such that
        \begin{align*}
            na_n^{-\beta}l(a_n)\to 1.
        \end{align*}
    \end{lemma}
    
    \begin{lemma}[Page 1525 of \citealt{gine1997student}]\label{lem:odd_number_1}
        If $Z_1,Z_2,\cdots$ are i.i.d. regularly varying with $\alpha\geq 2$ and $\E Z_1=0$,  then
    \begin{align*}
        \E Y_1^{k_1}\cdots Y_r^{k_r}=o(p^{-r}),
    \end{align*}
    where $k_1+\cdots+k_r$ is even, and at least one of $k_1,\cdots,k_r$ equals to one.
    \end{lemma}
    
    \begin{lemma}\label{lem:summation_of_sigma}
        For any non-random $p\times p$ (complex) matrix $\A$, we have 
        \begin{align}\label{form:summation}
            \sum_{i_1,\cdots,i_r}^* a_{i_{j_1}i_{j_2}}\cdots a_{i_{j_{2k-1}}i_{j_{2k}}}=
            \begin{cases}
                O(p)\|\A\|^k,&r=2,3,\\
                O(p^{\frac{r+\lfloor\frac{r}{4}\rfloor}{2}})\|\A\|^k,&r\geq 4,
            \end{cases}
        \end{align}
        where $j_1,\cdots,j_{2k}\in\{1,\cdots,r\}$ and $j_{2l-1}\not=j_{2l}$ for $l=1,\cdots,k$.
    \end{lemma}
    \begin{proof}
        For simplicity, we denote $\A^{\otimes m}$ as the Hadamard product of $m$ matrices $\A$. When $r=2$, \eqref{form:summation} can only be 
        \begin{align*}
            \sum_{i_1,i_2}^*a_{i_1i_2}^k=\one_p\trans\A^{\otimes k}\one_p-\tr(\A^{\otimes k})=O(p)\|\A\|^k.
        \end{align*}
        When $r=3$, there are only two cases. The first is 
        \begin{align*}
            \sum_{i_1,i_2,i_3}^*a_{i_1i_2}^{k_1}a_{i_1i_3}^{k_2}=&\sum_{i_1,i_2,i_3=1}^pa_{i_1i_2}^{k_1}a_{i_1i_3}^{k_2}+O(p)\|\A\|^k=O(p)\|\A\|^k,
        \end{align*}
        where $k_1+k_2=k$. And the second is 
        \begin{align*}
            \sum_{i_1,i_2,i_3}^*a_{i_1i_2}^{k_1}a_{i_1i_3}^{k_2}a_{i_2i_3}^{k_3}=&\sum_{i_1,i_2,i_3=1}^pa_{i_1i_2}^{k_1}a_{i_1i_3}^{k_2}a_{i_2i_3}^{k_3}+O(p)\|\A\|^k=O(p)\|\A\|^k,
        \end{align*} 
        where $k_1+k_2+k_3=k$.

        For $r\geq 4$, we combine all the $a_{ij}$s with the same indexes, that is 
        \begin{align}
            S:=\sum_{i_1,\cdots,i_r}^* a_{i_{j_1}i_{j_2}}^{k_1}\cdots a_{i_{j_{2u-1}}i_{j_{2u}}}^{k_u},\label{form:summation_of_sigma_without*}
        \end{align}
        where each $(j_{2l-1},j_{2l})$ is different from others and $k_1+\cdots+k_u=k$.
        By Cauchy's inequality,
        \begin{align*}
            |S|\leq p^{\frac{r}{2}}\sqrt{\sum_{i_1,\cdots,i_r} |a_{i_{j_1}i_{j_2}}|^{2k_1}\cdots |a_{i_{j_{2u-1}}i_{j_{2u}}}|^{2k_u}}:=p^{\frac{r}{2}}\sqrt{\widetilde{S}}.
        \end{align*}
        Without loss of generality, we take $k_1=\cdots=k_u=1$ and $u=k$, since that $k_i\geq 2$ have no effect on the power of $p$. We will show that this summation correponds to a graph. 
        Consider the different indexes $i_1,\cdots,i_r$ as $r$ different vertices and $ a_{ij}$ as an undirected edge between vertex $i$ and vertex $j$. So we define a graph $G=G(S):=\left(V,E\right)$, where $V=\{1,\cdots,r\}$ is the vertex set and the edge set $E$ is constructed by these $ a_{ij}$s. 
        Concretly, $ a_{i_si_t}$ correponds to an edge $e=(s,t)\in E$. For simplicity, we consider the connected graph and disconnected cases can be easily deduced by connected cases. 
        By the depth-first search algorithm, given a starting point $q_1$, we have a search queue of all the vertices of the connected graph 
        \begin{align*}
            Q=(q_1,\cdots,q_r).
        \end{align*}
        And the search queue has a nice property that for any $t\geq2$, there exists $1\leq s\leq t-1$ such that $q_s$ and $q_t$ are adjacent. So we denote
        \begin{align*}
            V_t=\{s:1\leq s<t,~(q_s,q_t)\in E\},\quad t\geq 2,
        \end{align*}
        and the summation can be written as 
        \begin{align*}
            \widetilde{S}=\sum_{i_{q_1}}\sum_{i_{q_2}}\prod_{s\in V_2}|a_{i_{q_s}i_{q_2}}|^{2}\sum_{i_{q_3}}\prod_{s\in V_3}|a_{i_{q_s}i_{q_3}}|^2\cdots\sum_{i_{q_r}}\prod_{s\in V_r}|a_{i_{q_s}i_{q_r}}|^2.
        \end{align*}
        Since that for $t\geq 2$, $V_t$ is not empty, 
        \begin{align*}
            \sum_{i_{q_t}}\prod_{s\in V_t}|a_{i_{q_s}i_{q_r}}|^2\leq\|\A\|^{2|V_t|}.
        \end{align*}
        So we have 
        \begin{align*}
            \widetilde{S}\leq\sum_{i_{q_1}}\|\A\|^{k}=\|\A\|^{2k}p,
        \end{align*}
        and 
        \begin{align*}
            |S|\leq p^{\frac{r}{2}}\sqrt{\widetilde{S}}\leq\|\A\|^kp^{\frac{r+1}{2}}.
        \end{align*}
        If the graph $G$ is disconnected, then we can divide it into different connected sub-graph $G=\cup_{i=1}^mG_i$. We denote $H_j=\{i:|G_i|=j\}$, and the order of $p$ is 
        \begin{align*}
            |H_2|+|H_3|+\sum_{j\geq 4}\frac{j|H_j|+1}{2}=\frac{1}{2}\sum_{j\geq 2}j|H_j|-\frac{|H_3|}{2}+\frac{1}{2}\sum_{j\geq 4}1\leq\frac{1}{2}\left(r+\lfloor\frac{r}{4}\rfloor\right).
        \end{align*}
    \end{proof}
    
    \begin{lemma}\label{lem:non_identity_diagonal_self_normalization}
        Suppose $(Z_i)_{i=1}^\infty$ are i.i.d. and regularly varying with index $\alpha\geq2$, $(a_i)_{i=1}^\infty$ is a positive sequence with $\inf a_i>0$. Then we have 
        \begin{align*}
            \E\frac{Z_1^{k_1}\cdots Z_r^{k_r}}{\left(a_1Z_1^2+\cdots+a_pZ_p^2\right)^{\frac{k}{2}}}=o(p^{-r+\delta})
        \end{align*}
        for any $\delta>0$.
    \end{lemma}
    \begin{proof}
        The expectation can be formulated as
        \begin{align*}
            \E\frac{Z_1^{k_1}\cdots Z_r^{k_r}}{\left(a_1Z_1^2+\cdots+a_pZ_p^2\right)^{\frac{k}{2}}}=\frac{1}{\Gamma(\frac{k}{2})}\int_0^\infty s^{\frac{k}{2}-1}\prod_{i=1}^r\E Z_i^{k_i}e^{-sa_iZ_i^2}\prod_{i=r+1}^p\varphi(a_is)ds.
        \end{align*}
        Denote $a=\inf a_i>0$, and we have $\varphi(a_is)\leq\varphi(as)$ for all $i$. Since that $Z_1$ is regularly varying with index $\alpha\geq2$, $\E|Z_1|^{2-\delta}<\infty$ for any small $\delta>0$. If $k_i=1$, since that for some random variable $\theta$ uniformly distributed on $[0,1]$ and independent of $Z_i$, 
        \begin{align*}
            \E Z_ie^{-sa_iZ_i^2}=-s\E Z_i^3e^{-sa_i\theta Z_i^2},
        \end{align*}
        so we have 
        \begin{align*}
            \left|\E Z_ie^{-sa_iZ_i^2}\right|\leq s\E|Z_i|^{2-\delta}\cdot|Z_i|^{1+\delta}e^{-sa_i\theta Z_i^2}\lesssim s(sa_i)^{-\frac{1+\delta}{2}}\leq a^{-\frac{1+\delta}{2}}s^{\frac{1-\delta}{2}},
        \end{align*}
        by the fact that $f(t)=t^{1+\delta}e^{st^2}\lesssim s^{-\frac{1+\delta}{2}}$ for all $t>0$. If $k_i\geq 2$, we have
        \begin{align*}
            \left|\E Z_i^{k_i}e^{-sa_iZ_i^2}\right|\leq\E|Z_i|^{2-\delta}\cdot|Z_i|^{k_i-2+\delta}e^{-sa_iZ_i^2}\lesssim a^{1-\frac{k_i+\delta}{2}}s^{1-\frac{k_i+\delta}{2}}.
        \end{align*}
        Denote $I_1=\{i:k_i=1\}$, $I_2=\{i:k_i\geq 2\}$, and we have 
        \begin{align*}
            \int_0^{\varepsilon}s^{\frac{k}{2}-1}\cdot s^{\frac{1-\delta}{2}|I_1|}\cdot s^{2|I_2|-\sum_{i\in I_2}\frac{k_i+\delta}{2}}\varphi^{p-r}(as)ds=\int_0^\varepsilon s^{(1-\frac{\delta}{2})r-1}\varphi^{p-r}(as)ds.
        \end{align*}
        By Lemma \ref{lem:norming_constants}, we can choose a sequence $(b_p)_{p=1}^\infty$ such that $pb_p^{-1}l(b_p)\to 1$ and
        \begin{align*}
            \varphi^p(\frac{t}{b_p})\sim e^{-t}.
        \end{align*}
        So we take a substitution as $t=ab_ps$ and then 
        \begin{align*}
            \int_0^\varepsilon s^{(1-\frac{\delta}{2})r-1}\varphi^{p-r}(as)ds=&\int_0^{\varepsilon b_p}\left(\frac{t}{ab_p}\right)^{(1-\frac{\delta}{2})r-1}\varphi^{p-r}(\frac{t}{b_p})\frac{1}{ab_p}dt\\
            \lesssim & b_p^{-(1-\frac{\delta}{2})r}\int_0^\infty t^{(1-\frac{\delta}{2})r-1}e^{-t}dt\lesssim b_p^{-(1-\frac{\delta}{2})r}.
        \end{align*}
        By the property of slowly varying function, for any $\gamma>0$, we have
        \begin{align*}
            p^{(1-\frac{\delta}{2})r-\gamma}b_p^{-(1-\frac{\delta}{2})r}=\left[pb_p^{-1}l(b_p)\right]^{(1-\frac{\delta}{2})r-\gamma}\cdot b_p^{-\gamma}l(b_p)^{-(1-\frac{\delta}{2})r+\gamma}\to 0.
        \end{align*}
        Since $\delta$ and $\gamma$ can be chosen arbitrarily small, so the proof is complete.
    \end{proof}

    \begin{lemma}\label{lem:alpha_4_diagonal}
        Suppose $(Z_i)_{i=1}^\infty$ are i.i.d. and regularly varying with index $\alpha>4$, $(a_i)_{i=1}^\infty$ is a positive sequence with $\inf a_i>0$. We denote 
        \begin{gather*}
            J_1=\{1\leq i\leq r:k_i=1\},\quad J_2=\{1\leq i\leq r:2\leq k_i\leq4\},\\
            \quad J_3=\{1\leq i\leq r:k_i\geq 5\},
        \end{gather*}
        then for any $4<\beta<\min\{\alpha,5\}$,
        \begin{align*}
            \E\frac{Z_1^{k_1}\cdots Z_r^{k_r}}{\left(a_1Z_1^2+\cdots+a_pZ_p^2\right)^{\frac{k}{2}}}=O(p^{-\frac{3}{2}|J_1|-\sum_{i\in J_2}\frac{k_i}{2}-\frac{\beta}{2}|J_3|}).
        \end{align*}
    \end{lemma}
    \begin{proof}
        We write
        \begin{align*}
            \E\frac{Z_1^{k_1}\cdots Z_r^{k_r}}{\left(a_1Z_1^2+\cdots+a_pZ_p^2\right)^{\frac{k}{2}}}=\frac{1}{\Gamma(\frac{k}{2})}\int_0^\infty s^{\frac{k}{2}-1}\prod_{i=1}^r\E Z_i^{k_i}e^{-sa_iZ_i^2}\prod_{i=r+1}^p\varphi(a_is)ds.
        \end{align*}
        For $i\in J_1$, we consider $f(s)=\E Z_ie^{-sa_iZ_i^2}$. Since that $f(0)=0$, by Taylor's expansion we have
        \begin{align*}
            |f(s)|=|f(0)+sf'(\xi)|\lesssim s.
        \end{align*}
        For $i\in J_2$, we have
        \begin{align*}
            \left|\E Z_i^{k_i}e^{-sa_iZ_i^2}\right|\leq\E|Z_i|^{k_i}\lesssim 1.
        \end{align*}
        For $i\in J_3$, since that $g(t)=t^{k_i-\beta}e^{-st^2}\lesssim s^{-\frac{k_i-\beta}{2}}$ for all $t>0$, so we have
        \begin{align*}
            \left|\E Z_i^{k_i}e^{-sa_iZ_i^2}\right|\leq\E|Z_i|^{\beta}\cdot|Z_i|^{k_i-\beta}e^{-sa_iZ_i^2}\lesssim s^{-\frac{k_i-\beta}{2}}.
        \end{align*}
        Combine all the cases above and $\varphi(a_is)\leq\varphi(as)$, we conclude that 
        \begin{align*}
            &\left|\E\frac{Z_1^{k_1}\cdots Z_r^{k_r}}{\left(a_1Z_1^2+\cdots+a_pZ_p^2\right)^{\frac{k}{2}}}\right|\\
            \lesssim&\int_0^\varepsilon s^{\frac{k}{2}-1}\cdot s^{|J_1|}\prod_{i\in J_3}s^{-\frac{k_i-\beta}{2}}\varphi^{p-r}(as)ds\\
            =&\int_0^\varepsilon s^{\frac{3}{2}|J_1|+\sum_{i\in J_2}\frac{k_i}{2}+\frac{\beta}{2}|J_3|-1}\varphi^{p-r}(as)ds\\
            =&\int_0^{\varepsilon ap} \left(\frac{t}{ap}\right)^{\frac{3}{2}|J_1|+\sum_{i\in J_2}\frac{k_i}{2}+\frac{\beta}{2}|J_3|-1}\varphi^{p-r}(\frac{t}{p})\frac{1}{ap}dt\\
            \lesssim &p^{-\frac{3}{2}|J_1|-\sum_{i\in J_2}\frac{k_i}{2}-\frac{\beta}{2}|J_3|}\int_0^\infty t^{\frac{3}{2}|J_1|+\sum_{i\in J_2}\frac{k_i}{2}+\frac{\beta}{2}|J_3|-1}e^{-t}dt\\
            \lesssim &p^{-\frac{3}{2}|J_1|-\sum_{i\in J_2}\frac{k_i}{2}-\frac{\beta}{2}|J_3|}.
        \end{align*}
    \end{proof}
  
    \section{Proofs of Self-normalization Propositions}
    \subsection{Proof of Proposition \ref{lem:even_number}}
    \begin{proof}
     With the Laplace transformation, the expectation can be written as
       \begin{align*}
           p^r\cdot\E Y_1^{2k_1}\cdots Y_r^{2k_r}=\frac{(-1)^kp^r}{\Gamma(k)}\int_0^\infty s^{k-1}\varphi^{(k_1)}(s)\cdots\varphi^{(k_r)}(s)\varphi^{p-r}(s)ds,
       \end{align*}
       where $k=k_1+\cdots+k_r$.
       For any fixed $\varepsilon>0$, we divide the integration into two parts as $\int_0^\varepsilon+\int_\varepsilon^\infty$.
    
 For $s \geq \varepsilon$, 
       \begin{align*}
           & p^r\int_\varepsilon^\infty s^{k-1}\varphi^{(k_1)}(s)\cdots\varphi^{(k_r)}(s)\varphi^{p-r}(s)ds\\
           \leq & p^r\varphi^{p-r}(\varepsilon)\int_\varepsilon^\infty s^{k-1}\varphi^{(k_1)}(s)\cdots\varphi^{(k_r)}(s)ds.
       \end{align*}
    Since $0 \leq \varphi(\varepsilon)<1$, we can get 
    \begin{align*}
        p^r\varphi^{p-r}(\varepsilon)\to 0.
    \end{align*}
    Noting
       \begin{align*}
        \varphi^{(k_1)}(s)\cdots\varphi^{(k_r)}(s)=(-1)^k \E Z_1^{2k_1}\cdots Z_r^{2k_r}e^{-s(Z_1^2+\cdots+Z_r^2)},
       \end{align*}
    we have
    \begin{align*}
    \left|\int_\varepsilon^\infty s^{k-1}\varphi^{(k_1)}(s)\cdots\varphi^{(k_r)}(s)ds \right|\leq & \int_0^\infty s^{k-1} \E Z_1^{2k_1}\cdots Z_r^{2k_r}e^{-s(Z_1^2+\cdots+Z_r^2)}ds\\
    =& \Gamma(k)\E \frac{Z_1^{2k_1}\cdots Z_r^{2k_r}}{(Z_1^2+\cdots+Z_r^2)^k}\leq \Gamma(k).
    \end{align*}   
Combing the two pieces, we can conclude
\begin{align*}
    \frac{(-1)^kp^r}{\Gamma(k)}\int_\epsilon^\infty s^{k-1}\varphi^{(k_1)}(s)\cdots\varphi^{(k_r)}(s)\varphi^{p-r}(s)ds \to 0.
\end{align*}

Next, we study 
\begin{align*}
    \frac{(-1)^kp^r}{\Gamma(k)}\int_0^\epsilon s^{k-1}\varphi^{(k_1)}(s)\cdots\varphi^{(k_r)}(s)\varphi^{p-r}(s)ds.
\end{align*}
and consider three cases. 
\begin{itemize}
    \item Case 1: $\alpha>2$. When $s\to 0^+$, by Lemma \ref{lem:exact_limiting}, 
\begin{align*}
    (-1)^k\varphi^{(k)}(s)\sim\begin{cases}
        \frac{\alpha}{2}\Gamma(k-\frac{\alpha}{2})s^{\frac{\alpha}{2}-k}l(\frac{1}{s}),&2k>\alpha,\\
        \E Z^{2k},&2k\leq\alpha,\E Z^{2k}<\infty,
    \end{cases}
\end{align*}
and by Lemma \ref{lem:exact_limiting_1-phi},
\begin{align*}
   1-\varphi(s) \sim  \E Z^{2}s.
\end{align*}

For $k_1=\cdots=k_r=1$, we can get 
\begin{align*}
    &  \frac{(-1)^kp^r}{\Gamma(k)}\int_0^\epsilon s^{k-1}\left[\varphi'(s)\right]^k\varphi^{p-r}(s)ds\\
    \sim & \frac{p^r}{\Gamma(k)}\int_0^\epsilon s^{k-1}(\E Z^2)^k (1-\E Z^{2}s)^{p-r}ds\\
    =&\frac{1}{\Gamma(k)}\int_0^{p\varepsilon  \E Z^2} t^{k-1}\left(1-\frac{t}{p}\right)^{p-r}dt,~\quad (t=p s \E Z^2)\\
    \to &\frac{1}{\Gamma(k)} \int_0^\infty t^{k-1}e^{-t}dt=1.
\end{align*}
 Otherwise, we have the bound
\begin{align*}
    \left|\varphi^{(k)}(s)\right|=&\left|\E Z^{2k}e^{-sZ^2}\right|\leq\E Z^{2(1+\delta)}\cdot Z^{2(k-1-\delta)} e^{-s Z^2}\\
    &\leq c_k \E Z^{2(1+\delta)} s^{1+\delta-k},
\end{align*}
where $0<\delta<\alpha-2$ and we used the fact $x^k e^{-x} \leq c_k=(k-1)^{k-1}e^{1-k}$. Denote $I_1=\{i:k_i=1\}$ and $I_2=\{i:k_i\geq 2\}$. If there exists some $k_i \geq 2$, then $|I_2|\geq 1$ and we can conclude
\begin{align*}
    & \left| \frac{(-1)^kp^r}{\Gamma(k)}\int_0^\epsilon s^{k-1} \varphi^{(k_1)}(s)\cdots\varphi^{(k_r)}(s)  \varphi^{p-r}(s)ds \right|\\
 \leq  & \frac{c p^r}{\Gamma(k)}\int_0^\epsilon s^{r-1+|I_2|\delta} (\E Z^{2})^{|I_1|}(\E Z^{2(1+\delta)})^{|I_2|} (1-\E Z^{2}s)^{p-r}ds\\
=&\frac{c (p \E Z^2)^{-|I_2|\delta}}{\Gamma(k)}\int_0^{p\epsilon  \E Z^2} t^{r-1+|I_2|\delta}\left(1-\frac{t}{p}\right)^{p-r}dt,~\quad (t=p s \E Z^2)\\
    \to & 0.
\end{align*}
    \item Case 2: $\alpha =2$ and $\E Z_1^2<\infty$. In this case, when $x\to\infty$, we must have $l(x)\to 0$ since 
    \begin{align*}
        \E Z^2=\int_0^\infty P(Z^2>t)dt\leq c_1+ c_2\int_1^\infty\frac{l(x)}{x}dx<\infty.
    \end{align*}
    Then we have
    \begin{align*}
        &(-1)^kp^r\int_0^\varepsilon s^{k-1}\varphi^{(k_1)}(s)\cdots\varphi^{(k_r)}(s)\varphi^{p-r}(s)ds\\
        \sim & p^r\int_0^\varepsilon s^{k-1}(\E Z^2)^{|I_1|}\prod_{i\in I_2}\left(s^{1-k_i}l(\frac{1}{s})\right)(1-\E Z^{2}s)^{p-r}ds\\
        = & p^r\int_0^\varepsilon s^{r-1}(\E Z^2)^{|I_1|}\left[l(\frac{1}{s})\right]^{|I_2|}(1-\E Z^{2}s)^{p-r}ds\\
        = & \int_0^{p\varepsilon\E Z^2}\left[l(\frac{p}{t\E Z^2})\right]^{|I_2|}t^{r-1}\left(1-\frac{t}{p}\right)^{p-r}dt\\
        \sim & \left[l(p)\right]^{|I_2|}\int_0^\infty t^{r-1}e^{-t}dt\\
        \to&\begin{cases}
            \Gamma(r),&|I_2|=0,\\
            0,&|I_2|\geq 1.
        \end{cases}
       \end{align*}
    \item Case 3: $\alpha =2$ and $\E Z_1^2=\infty$. When $s\to 0^+$, by Lemma \ref{lem:exact_limiting}, 
    \begin{align*}
        (-1)^k\varphi^{(k)}(s)\sim\begin{cases}
            \Gamma(k-1)s^{1-k}l(\frac{1}{s}),&k>1,\\
            \widetilde{l}(\frac{1}{s}),&k=1,\\
        \end{cases}\quad(s\to 0^+)
    \end{align*}
    and by Lemma \ref{lem:exact_limiting_1-phi},
    \begin{align*}
        1-\varphi(s) \sim s \widetilde{l}\left(\frac{1}{s}\right).
    \end{align*}
    $\widetilde{l}(x)$ is a slowly varying function. By Lemma \ref{lem:norming_constants}, there exists a positive sequence $(a_n)_{n=1}^\infty$ such that 
    \begin{align*}
        na_n^{-1}\widetilde{l}(a_n)\to 1.
    \end{align*}
    Take a substitution as $s=t/a_p$, we observe that
       \begin{align*}
           \varphi^p\left(\frac{t}{a_p}\right)=&\exp\left\{p\log\varphi(\frac{t}{a_p})\right\}\sim\exp\left\{p\left(\varphi(\frac{t}{a_p})-1\right)\right\}\\
           \sim& \exp\left\{-\frac{pt}{a_p} \widetilde{l}\left(\frac{a_p}{t}\right)\right\}=\exp\left\{-t \cdot \frac{p}{a_p} \widetilde{l}(a_p) \cdot \frac{\widetilde{l}\left(\frac{a_p}{t}\right)}{\widetilde{l}(a_p)}  \right\}\sim e^{-t}.
       \end{align*}
       And then
       \begin{align*}
        &(-1)^kp^r\int_0^\varepsilon s^{k-1}\varphi^{(k_1)}(s)\cdots\varphi^{(k_r)}(s)\varphi^{p-r}(s)ds\\
        =&(-1)^kp^r\int_0^{\varepsilon a_p}\left(\frac{t}{a_p}\right)^{k-1}\varphi^{(k_1)}(\frac{t}{a_p})\cdots\varphi^{(k_r)}(\frac{t}{a_p})\varphi^{p-r}(\frac{t}{a_p})\frac{1}{a_p}dt\\
        \sim&p^r\int_0^{\varepsilon a_p}\frac{t^{k-1}}{a_p^k}\left[\widetilde{l}(\frac{a_p}{t})\right]^{|I_1|}\prod_{i\in I_2}\frac{\Gamma(k_i-1)t^{1-k_i}l(\frac{a_p}{t})}{a_p^{1-k_i}}e^{-t}dt\\
        \sim&\frac{p^r[\widetilde{l}(a_p)]^{|I_1|}[l(a_p)]^{|I_2|}}{a_p^r}\int_0^\infty t^{r-1}e^{-t}dt\\
        =&\frac{p^r[\widetilde{l}(a_p)]^{r}}{a_p^r}\cdot\left[\frac{l(a_p)}{\widetilde{l}(a_p)}\right]^{|I_2|}\cdot\Gamma(r)\\
        \to&\begin{cases}
            \Gamma(r),&|I_2|=0,\\
            0,&|I_2|\geq 1,
        \end{cases}
       \end{align*}
       where the last convergence holds by $l(a_p)/\widetilde{l}(a_p)\to 0$ and more details can be found in Theorem A.7 of \citet{fuchs2002expectation}.
\end{itemize}

    \end{proof}

    \subsection{Proof of Proposition \ref{lem:odd_number_2}}
    \begin{proof}
        By Lemma \ref{lem:odd_number_1} \citet{gine1997student}, the conclusion holds for the case where at least one of $k_1,\cdots,k_r$ equals to one.  We extend the result to general case where at least one of $k_1,\cdots,k_r$ is odd by induction.
        
        Without loss of generality, we assume $k_1$ is odd and set $k_1=2m+1$. If $m=0$, the result is true by Lemma \ref{lem:odd_number_1}. Next, we assume the conclusion holds for $m$, i.e.,
        \begin{align*}
            \E Y_1^{2m+1} Y_2^{k_2} \cdots Y_r^{k_r}=o(p^{-r}),
        \end{align*}
        if $k_2+\cdots+k_r$ is odd. 
        
        Noting $\sum_{i=1}^p Y_i^2=1$, we have
                \begin{align*}
            & \E Y_1^{2m+1} Y_2^{k_2} \cdots Y_r^{k_r}=\E Y_1^{2m+1} Y_2^{k_2} \cdots Y_r^{k_r} \cdot \sum_{i=1}^pY_i^2\\
                    =&\E Y_1^{2m+3} Y_2^{k_2} \cdots Y_r^{k_r}+\sum_{i=2}^r\E Y_1^{2m+1}\cdots Y_i^{k_i+2}\cdots Y_r^{k_r}\\
                    &+\sum_{i=r+1}^p \E Y_1^{2m+1} Y_2^{k_2} \cdots Y_r^{k_r} Y_i^2\\
                    =&\E Y_1^{2(m+1)+1} Y_2^{k_2} \cdots Y_r^{k_r}+(r-1) o(p^{-r})+(p-r)\cdot o\left(p^{-(r+1)}\right)
                \end{align*}
        which yields 
        \begin{align*}
            \E Y_1^{2(m+1)+1} Y_2^{k_2} \cdots Y_r^{k_r}=o(p^{-r}).
        \end{align*}   
        That is the conclusion is true for $m+1$. The proof is completed. 
            \end{proof}

    \subsection{Proof of Proposition \ref{prop:even_number4}}
    \begin{proof}
        With the same routine in Proposition \ref{lem:even_number}, we only need to study
        \begin{align*}
            \frac{(-1)^kp^r}{\Gamma(k)}\int_0^\epsilon s^{k-1}\varphi^{(k_1)}(s)\cdots\varphi^{(k_r)}(s)\varphi^{p-r}(s)ds.
        \end{align*}
        When $s\to 0^+$, by Lemma \ref{lem:exact_limiting}, 
\begin{align*}
    (-1)^k\varphi^{(k)}(s)\sim\begin{cases}
        \frac{\alpha}{2}\Gamma(k-\frac{\alpha}{2})s^{\frac{\alpha}{2}-k}l(\frac{1}{s}),&2k>\alpha,\\
        \E Z^{2k},&2k\leq\alpha,\E Z^{2k}<\infty,
    \end{cases}
\end{align*}
and by Lemma \ref{lem:exact_limiting_1-phi},
\begin{align*}
   1-\varphi(s) \sim  \E Z^{2}s.
\end{align*}
If $|I_3|=0$, we have
\begin{align*}
    & (-1)^kp^{|I_1|+2|I_2|}\int_0^\epsilon s^{k-1}\left[\varphi'(s)\right]^{|I_1|}\left[\varphi''(s)\right]^{|I_2|}\varphi^{p-r}(s)ds\\
    \sim& p^{|I_1|+2|I_2|}\int_0^\epsilon s^{k-1}(\E Z^2)^{|I_1|}(\E Z^4)^{|I_2|}\left(1-\E Z^2s\right)^{p-r}ds\\
    =& (\E Z^2)^{-2|I_2|}(\E Z^4)^{|I_2|}\int_0^{\varepsilon p\E Z^2}t^{|I_1|+2|I_2|-1}\left(1-\frac{t}{p}\right)^{p-r}dt\\
    \to & (\E Z^2)^{-2|I_2|}(\E Z^4)^{|I_2|}\Gamma(k).
\end{align*}
Otherwise, we have the bound 
\begin{align*}
    \left|\varphi^{(k)}(s)\right|=\left|\E Z^{2k}e^{-sZ^2}\right|\leq\E Z^{\gamma}\cdot Z^{2k-\gamma} e^{-s Z^2}\leq c_k \E Z^{\gamma} s^{\frac{\gamma}{2}-k},
\end{align*}
where $k\geq 3$ and $\beta<\gamma<\min\{\alpha,6\}$. If $|I_3|\geq 1$, then we can conclude
\begin{align*}
    & \left|(-1)^kp^{|I_1|+2|I_2|+\frac{\beta}{2}|I_3|}\int_0^\epsilon s^{k-1}\left[\varphi'(s)\right]^{|I_1|}\left[\varphi''(s)\right]^{|I_2|}\prod_{i\in I_3}\varphi^{(k_i)}(s)\varphi^{p-r}(s)ds\right|\\
    \leq& cp^{|I_1|+2|I_2|+\frac{\beta}{2}|I_3|}\int_0^\epsilon s^{|I_1|+2|I_2|+\frac{\gamma}{2}|I_3|-1}\left(1-\E Z^2s\right)^{p-r}ds\\
    =&cp^{\frac{\gamma-\beta}{2}|I_3|}(\E Z^2)^{|I_1|}(\E Z^4)^{|I_2|}(\E Z^\gamma)^{|I_3|}\int_0^{p\epsilon  \E Z^2} t^{|I_1|+2|I_2|+\frac{\gamma}{2}|I_3|-1}\left(1-\frac{t}{p}\right)^{p-r}dt\\
    \to&0.
\end{align*}
    \end{proof}
    
\subsection{Proof of Proposition \ref{prop:alpha>4_general_number}}
\begin{proof}[Proof of Proposition \ref{prop:alpha>4_general_number}]
    We write 
    \begin{align*}
        \E Y_1^{k_1}\cdots Y_r^{k_r}=\frac{1}{\Gamma(\frac{k}{2})}\int_0^\infty s^{\frac{k}{2}-1}\prod_{i=1}^r\E Z_i^{k_i}e^{-sZ_i^2}\varphi^{p-r}(s)ds.
    \end{align*}
    For $j\in J_1$, we consider $f(s)=\E Z_ie^{-sZ_i^2}$. Since that $f(0)=0$, by Taylor's expansion we have
    \begin{align*}
        |f(s)|=|f(0)+sf'(\xi)|\lesssim s.
    \end{align*}
    For $j\in J_2$, we have
    \begin{align*}
        \left|\E Z_i^{k_i}e^{-sZ_i^2}\right|\leq\E|Z_i|^{k_i}\lesssim 1.
    \end{align*}
    For $j\in J_3$, since that $g(t)=t^{k_i-\beta}e^{-st^2}\lesssim s^{-\frac{k_i-\beta}{2}}$ for all $t>0$, so we have
    \begin{align*}
        \left|\E Z_i^{k_i}e^{-sZ_i^2}\right|\leq\E|Z_i|^{\beta}\cdot|Z_i|^{k_i-\beta}e^{-sZ_i^2}\lesssim s^{-\frac{k_i-\beta}{2}}.
    \end{align*}
    Combine all the cases above, we conclude that 
    \begin{align*}
        \left|\E Y_1^{k_1}\cdots Y_r^{k_r}\right|\lesssim&\int_0^\varepsilon s^{\frac{k}{2}-1}\cdot s^{|J_1|}\prod_{j\in J_3}s^{-\frac{k_i-\beta}{2}}\varphi^{p-r}(s)ds\\
        =&\int_0^\varepsilon s^{\frac{3}{2}|J_1|+\sum_{j\in J_2}\frac{k_j}{2}+\frac{\beta}{2}|J_3|-1}\varphi^{p-r}(s)ds\\
        =&\int_0^{\varepsilon p} \left(\frac{t}{p}\right)^{\frac{3}{2}|J_1|+\sum_{i\in J_2}\frac{k_i}{2}+\frac{\beta}{2}|J_3|-1}\varphi^{p-r}(\frac{t}{p})\frac{1}{p}dt\\
        \sim &p^{-\frac{3}{2}|J_1|-\sum_{i\in J_2}\frac{k_i}{2}-\frac{\beta}{2}|J_3|}\int_0^\infty t^{\frac{3}{2}|J_1|+\sum_{i\in J_2}\frac{k_i}{2}+\frac{\beta}{2}|J_3|-1}e^{-t}dt\\
        \lesssim &p^{-\frac{3}{2}|J_1|-\sum_{i\in J_2}\frac{k_i}{2}-\frac{\beta}{2}|J_3|}.
    \end{align*}
\end{proof}

 \section{Proofs of Quadratics Propositions}   
 \subsection{Proof of Proposition \ref{prop:moments_of_self-normalized_quadratic_forms}}
    \begin{proof}
By Propositions \ref{lem:even_number} and \ref{lem:odd_number_2}, 
        \begin{align*}
            \E\Y\trans\A\Y=&\sum_{i}a_{ii}\E Y_i^2+\sum_{i\not=j}a_{ij}\E Y_iY_j
            =(\E Y_1^2-\E Y_1 Y_2) \tr\A+ \E Y_1 Y_2 \one_p \trans \A \one_p\\
             =& \frac{\tr \A}{p}+o\left(\frac{1}{p}\right) \|\A\|.
        \end{align*}
Next, we consider the variance of the quadratic forms. Similar to the analysis of Hanson-Wright inequality, we study the diagonal and off-diagonal sums separately. For the diagonal sums,
\begin{align*}
\var\left( \sum_{i=1}^p a_{ii} Y_i^2 \right)=&\sum_{i=1}^p a^2_{ii} \var\left(Y_i^2\right)+\sum_{i \neq j} a_{ii} a_{jj} \cov(Y_i^2,Y_j^2)\\
=& \left(\var\left(Y_1^2\right)-\cov(Y_1^2,Y_2^2) \right)  \sum_{i=1}^p a^2_{ii}+\cov(Y_1^2,Y_2^2)\sum_{i,j} a_{ii} a_{jj}\\
=&\frac{o(1)}{p} \tr(\A \circ \A)+\frac{o(1)}{p^2}\tr^2(\A)=o(1) \|\A\|^2.
\end{align*}  
For the off-diagonal sums, denoting 
\begin{align*}
    \tilde{\A}=\A-\diag(\A)=(\tilde{a}_{ij})_{p \times p},
\end{align*}
we have
\begin{align*}
    &\var\left( \sum_{i \neq j} a_{ij} Y_i Y_j \right)= \cov \left( \sum_{i \neq j} a_{ij} Y_i Y_j,\sum_{k \neq l} a_{kl} Y_k Y_l \right)\\
    =&\sum_{i \neq j} a_{ij} \cov \left(Y_iY_j,\sum_{k \neq l} a_{kl} Y_k Y_l\right)\\
    =&\sum_{i \neq j} a_{ij}\left(2 a_{ij} \var(Y_1Y_2)+4 \cov(Y_1Y_2,Y_1Y_3) \sum_{k \neq i} \tilde{a}_{jk}\right)\\
    &+\sum_{i \neq j} a_{ij}\left(\cov(Y_1Y_2,Y_3Y_4) \sum_{k,l \notin \{i,j\}} \tilde{a}_{kl}     \right)\\
    =&   \cov(Y_1Y_2,Y_3Y_4) \left(\sum_{i,j} \tilde{a}_{ij} \right)^2 +4 \left(\cov(Y_1Y_2,Y_1Y_3)- \cov(Y_1Y_2,Y_3Y_4) \right) \\& \cdot \sum_{i,j,k}\tilde{a}_{ij}\tilde{a}_{jk}
    +2 \left(\var(Y_1Y_2)-\cov(Y_1Y_2,Y_1Y_3)-\cov(Y_1Y_2,Y_3Y_4)\right)\sum_{i,j} \tilde{a}^2_{ij}. 
\end{align*}
By Propositions \ref{lem:even_number} and \ref{lem:odd_number_2}, 
\begin{gather*}
    \var(Y_1Y_2)= \frac{1+o(1)}{p^2},~\cov(Y_1Y_2,Y_1Y_3)=o(p^{-3}),\quad \cov(Y_1Y_2,Y_3Y_4)=o(p^{-4}).
\end{gather*}   
Noting 
\begin{gather*}
    \sum_{i,j} \tilde{a}_{ij}=\one_p\trans \tilde{\A} \one_p \leq  p \|\tilde{\A}\| \leq 2p \|\A\|,\\
    \sum_{i,j,k}\tilde{a}_{ij}\tilde{a}_{jk}=\one_p \trans \tilde{\A}^2 \one_p \leq p  \|\tilde{\A}\|^2 \leq 4  p \|\A\|^2,\\
    \sum_{i,j} \tilde{a}^2_{ij}=\tr  \tilde{\A}^2 \leq p \|\tilde{\A}\|^2 \leq 2p \|\A\|^2,
\end{gather*}   
we can claim
\begin{align*}
    &\var\left( \sum_{i \neq j} a_{ij} Y_i Y_j \right) \leq \frac{5}{p} \|\A\|^2
\end{align*}
for large enough $p$.  

Thus, 
\begin{align*}
   \var\left(\Y\trans\A\Y \right)\leq 2 \var\left( \sum_{i=1}^p a_{ii} Y_i^2 \right)+2 \var\left( \sum_{i \neq j} a_{ij} Y_i Y_j \right)=o(1) \|\A\|^2,
\end{align*}
and 
\begin{align*}
    \E \left|\Y\trans\A\Y- \frac{\tr \A}{p}\right|^2=\var\left(\Y\trans\A\Y \right)+\left|   \E \Y\trans\A\Y- \frac{\tr \A}{p} \right|^2=o(1) \|\A\|^2.
\end{align*}

Noting $\|\Y\|_2=1$, we can get 
\begin{align*}
    \left|\Y\trans\A\Y- \frac{\tr \A}{p}\right| \leq \left|\Y\trans\A\Y\right|^2+\|\A\| \leq 2 \|\A\|.
\end{align*}
Therefore, for $k>2$,
 \begin{align*}
    \E \left|\Y\trans\A\Y- \frac{\tr \A}{p}\right|^k \leq (2 \|\A\|)^{k-2} \E \left|\Y\trans\A\Y- \frac{\tr \A}{p}\right|^2=o(1) \|\A\|^k.
 \end{align*}
\end{proof}    

\subsection{Proof of Proposition \ref{prop:alpha>4_moments_of_self-normalized_quadratic_forms}}
\begin{proof}
    By Proposition \ref{prop:even_number4} and \ref{prop:alpha>4_general_number}
    \begin{align*}
        \E\Y\trans\A\Y=& \sum_{i}a_{ii}\E Y_i^2+\sum_{i\not=j}a_{ij}\E Y_iY_j=(\E Y_1^2-\E Y_1 Y_2) \tr\A+ \E Y_1 Y_2 \one_p \trans \A \one_p\\
        =& \frac{\tr \A}{p}+O(\frac{1}{p^2}) \|\A\|.
    \end{align*}
    Next, we consider the covariance of the quadratic forms. By Proposition \ref{prop:even_number4} and \ref{prop:alpha>4_general_number},
    \begin{align*}
        &\E\left(\Y\trans\A\Y-\frac{1}{p}\tr\A\right)\left(\Y\trans\B\Y-\frac{1}{p}\tr\B\right)\\
        =&\E\left(\sum_{i}a_{ii}\left(Y_i^2-\frac{1}{p}\right)+\sum_{i\not=j}a_{ij}Y_iY_j\right)\left(\sum_{i}b_{ii}\left(Y_i^2-\frac{1}{p}\right)+\sum_{i\not=j}b_{ij}Y_iY_j\right)\\
        =&\sum_{i,j}a_{ii}b_{jj}\E\left(Y_i^2-\frac{1}{p}\right)\left(Y_j^2-\frac{1}{p}\right)+\sum_{i;j\not=k}a_{ii}b_{jk}\E\left(Y_i^2-\frac{1}{p}\right)Y_jY_k\\
        &+\sum_{i\not=j;k}a_{ij}b_{kk}\E Y_iY_j\left(Y_k^2-\frac{1}{p}\right)+\sum_{i\not=j,k\not=l}a_{ij}b_{kl}\E Y_iY_jY_kY_l.\\
        =&\sum_{i}a_{ii}b_{ii}\left(\E Y_i^4-\frac{1}{p^2}\right)+\sum_{i\not=j}a_{ii}b_{jj}\left(\E Y_i^2Y_j^2-\frac{1}{p^2}\right)+2\sum_{i\not=j}a_{ij}b_{ij}\E Y_i^2Y_j^2\\
        &+O(\frac{1}{p^2})\|\A\|\|\B\|\\
        =&(\tau-3)\frac{\tr(\A\circ\B)}{p^2}+2\frac{\tr(\A\B)}{p^2}+(1-\tau)\frac{\tr\A\tr\B}{p^3}+o(\frac{1}{p})\|\A\|\|\B\|.
    \end{align*}
        In the last equality, we use the identity
        \begin{align*}
            p(p-1)\E Y_1^2Y_2^2+p\E Y_1^4=1
        \end{align*}
        to obtain
        \begin{align*}
            \E Y_i^2Y_j^2-\frac{1}{p^2}=\frac{1-p^2\E Y_1^4}{p^2(p-1)}=\frac{1-\tau}{p^3}+o(\frac{1}{p^3}).
        \end{align*}
        Finally, we consider the concentration inequality. Similar to the proof of Proposition \ref{prop:moments_of_self-normalized_quadratic_forms}, we only need to show that 
        \begin{align*}
            \E\left|\Y\trans\A\Y-\frac{1}{p}\tr\A\right|^4=o(\frac{1}{p})\|\A\|^4.
        \end{align*}
        For the diagonal sums, it can be decomposed to the following summation
        \begin{align}
            \sum_{i_1,\cdots,i_r}^*(a_{i_1i_1}-\frac{\tr\A}{p})^{k_1}\cdots(a_{i_ri_r}-\frac{\tr\A}{p})^{k_r}\E Y_{i_1}^{2k_1}\cdots Y_{i_r}^{2k_r},\label{form:alpha>4_item_1}
        \end{align}
        where $k_1+\cdots+k_r=4$. By Proposition \ref{prop:even_number4}, we have
        \begin{align}\label{form:alpha>4_item_1_order}
            \E Y_1^{2k_1}\cdots Y_r^{2k_r}=\begin{cases}
                O(p^{-|I_1|-2|I_2|}),&|I_3|=0,\\
                o(p^{-|I_1|-2|I_2|}),&|I_3|\geq 1.
            \end{cases}
        \end{align}
        By the fact that $\sum_{i}(a_{ii}-\tr\A/p)=0$, it is not difficult to verify that
        \begin{align}
            \sum_{i_1,\cdots,i_r}^*(a_{i_1i_1}-\frac{\tr\A}{p})^{k_1}\cdots(a_{i_ri_r}-\frac{\tr\A}{p})^{k_r}=O(p^{r-\lceil\frac{|I_1|}{2}\rceil})\|\A\|^k.\label{form:1-power_numbers}
        \end{align}
        Combined with \eqref{form:1-power_numbers} and \eqref{form:alpha>4_item_1_order}, we have 
        \begin{align*}
            \eqref{form:alpha>4_item_1}=O(p^{r-\lceil\frac{|I_1|}{2}\rceil-|I_1|-2|I_2|-\frac{\beta}{2}|I_3|})\|\A\|^k.
        \end{align*}
        Since that $r=|I_1|+|I_2|+|I_3|$, the order of $p$ is 
        \begin{align*}
            r-\lceil\frac{|I_1|}{2}\rceil-|I_1|-2|I_2|-\frac{\beta}{2}|I_3|=-\lceil\frac{|I_1|}{2}\rceil-|I_2|-(\frac{\beta}{2}-1)|I_3|<-1.
        \end{align*}
        So we deduce that $\eqref{form:alpha>4_item_1}=o(p^{-1})$.
        
        For the off-diagonal sums, it can be decomposed to the following summation 
        \begin{align}
            \sum_{i_1,\cdots,i_r}^*a_{i_{j_1}i_{j_2}}a_{i_{j_3}i_{j_4}}a_{i_{j_5}i_{j_6}}a_{i_{j_{7}}i_{j_{8}}}\E Y_{i_1}^{k_1}\cdots Y_{i_r}^{k_r},\label{form:alpha>4_item_2}
        \end{align}
        where $k_1+\cdots+k_r=8$, $j_1,\cdots,j_{8}\in\{1,\cdots,r\}$ and $j_{2l-1}\not=j_{2l}$ for $l=1,\cdots,4$. We consider the following three cases.
        \begin{itemize}
            \item Case 1:  When $r\leq 3$, by Lemma \ref{lem:summation_of_sigma}, 
            \begin{align*}
                \left|\sum_{i_1,\cdots,i_r}^*a_{i_{j_1}i_{j_2}}a_{i_{j_3}i_{j_4}}a_{i_{j_5}i_{j_6}}a_{i_{j_{7}}i_{j_{8}}}\right|=O(p)\|\A\|^4.
            \end{align*}
            And by Proposition \ref{prop:alpha>4_general_number}, we have 
            \begin{align*}
                \E Y_{i_1}^{k_1}\cdots Y_{i_r}^{k_r}=O(p^{-4}),
            \end{align*}
            which leads to $\eqref{form:alpha>4_item_2}=O(p^{-3}).$
            \item Case 2:  When $r=4$, if at least one of $k_1,k_2,k_3,k_4$ is odd, by Proposition \ref{prop:alpha>4_general_number},
            \begin{align*}
                \E Y_{i_1}^{k_1}Y_{i_2}^{k_2}Y_{i_3}^{k_3}Y_{i_4}^{k_4}=O(p^{-5}).
            \end{align*}
            And by Lemma \ref{lem:summation_of_sigma},
            \begin{align*}
                \left|\sum_{i_1,i_2,i_3,i_4}^*a_{i_{j_1}i_{j_2}}a_{i_{j_3}i_{j_4}}a_{i_{j_5}i_{j_6}}a_{i_{j_{7}}i_{j_{8}}}\right|=O(p^{\frac{5}{2}})\|\A\|^4.
            \end{align*}
            Otherwise, $k_1,k_2,k_3,k_4$ are even and equal to $2$. In this situation, by Proposition \ref{prop:even_number4},
            \begin{align*}
                \E Y_{i_1}^{2}Y_{i_2}^{2}Y_{i_3}^{2}Y_{i_4}^{2}=O(p^{-4}).
            \end{align*}
            And we must have 
            \begin{align*}
                \left|\sum_{i_1,i_2,i_3,i_4}^*a_{i_{j_1}i_{j_2}}a_{i_{j_3}i_{j_4}}a_{i_{j_5}i_{j_6}}a_{i_{j_{7}}i_{j_{8}}}\right|=O(p^{2})\|\A\|^4.
            \end{align*}
            Both two cases lead to $\eqref{form:alpha>4_item_2}=O(p^{-2})\|\A\|^4.$
            \item Case 3:  When $5\leq r\leq 8$, by Proposition \ref{prop:alpha>4_general_number},
            \begin{align*}
                \E Y_{i_1}^{k_1}\cdots Y_{i_r}^{k_r}=O(p^{-6}).
            \end{align*}
            And by Lemma \ref{lem:summation_of_sigma},
            \begin{align*}
                \left|\sum_{i_1,i_2,i_3,i_4}^*a_{i_{j_1}i_{j_2}}a_{i_{j_3}i_{j_4}}a_{i_{j_5}i_{j_6}}a_{i_{j_{7}}i_{j_{8}}}\right|=O(p^{3})\|\A\|^4.
            \end{align*}
            We conclude that $\eqref{form:alpha>4_item_2}=O(p^{-3})\|\A\|^4$.
        \end{itemize}
        Collecting all these cases, we have 
        \begin{align*}
            \left|\sum_{i\not=j}a_{ij}Y_iY_j\right|^4=O(\frac{1}{p^2})\|\A\|^4.
        \end{align*}
\end{proof}

\subsection{Proof of Proposition \ref{prop:general_self_normalization}}
\begin{proof}
    When $\alpha\geq 2$, by identity
    \begin{align}
        \frac{1}{\Y\trans\bSig\Y}=1-\frac{\Y\trans\bSig\Y-1}{\Y\trans\bSig\Y},\label{form:identity_1}
    \end{align}
    we have 
    \begin{align*}
        \E\frac{\Y\trans\A\Y}{\Y\trans\bSig\Y}=\E \Y\trans\A\Y-\E\frac{\Y\trans\A\Y\left(\Y\trans\bSig\Y-1\right)}{\Y\trans\bSig\Y}.
    \end{align*}
    By Proposition \ref{prop:moments_of_self-normalized_quadratic_forms}, 
    \begin{align*}
        \E\Y\trans\A\Y=\frac{\tr\A}{p}+o(\frac{1}{p})\|\A\|,
    \end{align*}
    and 
    \begin{align*}
        \left|\E\frac{\Y\trans\A\Y\left(\Y\trans\bSig\Y-1\right)}{\Y\trans\bSig\Y}\right|\leq\frac{\|\A\|}{\lambda_{\min}(\bSig)}\E\left|\Y\trans\bSig\Y-1\right|=o(1)\frac{\|\A\|\|\bSig\|}{\lambda_{\min}(\bSig)},
    \end{align*}
    so we can obtain 
    \begin{align*}
        \E\frac{\Y\trans\A\Y}{\Y\trans\bSig\Y}=\frac{\tr\A}{p}+o(1)\frac{\|\A\|\|\bSig\|}{\lambda_{\min}(\bSig)}.
    \end{align*}
    For concentration inequality, by \eqref{form:identity_1} and Proposition \ref{prop:moments_of_self-normalized_quadratic_forms},
    \begin{align*}
        \E\left|\frac{\Y\trans\A\Y}{\Y\trans\bSig\Y}-\frac{\tr\A}{p}\right|^k=&\E\left|\frac{\Y\trans\A\Y}{\Y\trans\bSig\Y}-\Y\trans\A\Y+\Y\trans\A\Y-\frac{1}{p}\tr\A\right|^k\\
        \lesssim&\E\left|\frac{\Y\trans\A\Y\left(\Y\trans\bSig\Y-1\right)}{\Y\trans\bSig\Y}\right|^k+\E\left|\Y\trans\A\Y-\frac{1}{p}\tr\A\right|^k\\
        =&o(1)\frac{\|\A\|^k\|\bSig\|^k}{\lambda_{\min}^{k}(\bSig)}.
    \end{align*}
    \bigskip

    When $\alpha>4$, 
    we have 
    \begin{align*}
        \E\frac{\Y\trans\A\Y}{\Y\trans\bSig\Y}=&\E\Y\trans\A\Y-\E\Y\trans\A\Y\left(\Y\trans\bSig\Y-1\right)\\
        &+\E\Y\trans\A\Y\left(\Y\trans\bSig\Y-1\right)^2-\E\frac{\Y\trans\A\Y\left(\Y\trans\bSig\Y-1\right)^3}{\Y\trans\bSig\Y}.
    \end{align*}
    By Proposition \ref{prop:alpha>4_moments_of_self-normalized_quadratic_forms}, we obtain that 
    \begin{align*}
        &\E\Y\trans\A\Y=\frac{\tr\A}{p}+O(p^{-2})\|\A\|,\\
        &\E\Y\trans\A\Y\left(\Y\trans\bSig\Y-1\right)\\
        =&\E\left(\Y\trans\A\Y-\frac{1}{p}\tr\A\right)\left(\Y\trans\bSig\Y-1\right)+\frac{1}{p}\tr\A \E\left(\Y\trans\bSig\Y-1\right)\\
        =&\frac{\tau-3}{p^2}\tr(\A\circ\bSig)+\frac{2}{p^2}\tr(\A\bSig)+\frac{1-\tau}{p^2}\tr\A+o(p^{-1})\|\A\|\|\bSig\|,\\
        &\E\Y\trans\A\Y\left(\Y\trans\bSig\Y-1\right)^2\\
        =&\frac{1}{p}\tr\A \E\left(\Y\trans\bSig\Y-1\right)^2+\E\left(\Y\trans\A\Y-\frac{1}{p}\tr\A\right)\left(\Y\trans\bSig\Y-1\right)^2\\
        =&\tr\A\left[\frac{\tau-3}{p^3}\tr(\bSig\circ\bSig)+\frac{2}{p^3}\tr(\bSig^2)+\frac{1-\tau}{p^2}\right]+o(p^{-1})\|\A\|\|\bSig\|^2,\\
        &\left|\E\frac{\Y\trans\A\Y\left(\Y\trans\bSig\Y-1\right)^3}{\Y\trans\bSig\Y}\right| \leq \frac{\|\A\|}{\lambda_{\min}(\bSig)} \E \left|\Y\trans\bSig\Y-1 \right|^3=o(p^{-1})\frac{\|\A\| \|\bSig\|^3}{\lambda_{\min}(\bSig)}.
    \end{align*}
Combining all the pieces, we can conclude
    \begin{align*}
        \E\frac{\Y\trans\A\Y}{\Y\trans\bSig\Y}=&\frac{\tr\A}{p}+\frac{\tau-3}{p^3}\left( \tr \A \tr(\bSig \circ \bSig)-p \tr(\A \circ \bSig)  \right)\\
        &+\frac{2}{p^3}\left(\tr \A \tr(\bSig^2)-p \tr(\A \bSig) \right)+o(p^{-1})\frac{\|\A\|\|\bSig\|^3}{\lambda_{\min}(\bSig)}.
    \end{align*}
Furthermore,
    \begin{align*}
        &\E\left(\frac{\Y\trans\A\Y}{\Y\trans\bSig\Y}-\frac{\tr\A}{p}\right)\left(\frac{\Y\trans\A\Y}{\Y\trans\bSig\Y}-\frac{\tr\B}{p}\right)\\
        =&\E\left(\frac{\Y\trans\A\Y}{\Y\trans\bSig\Y}-\Y\trans\A\Y+\Y\trans\A\Y-\frac{1}{p}\tr\A \right)\\
        &\cdot \left(\frac{\Y\trans\B\Y}{\Y\trans\bSig\Y}-\Y\trans\B\Y+\Y\trans\B\Y-\frac{1}{p}\tr\B\right)\\
        =&\E \left(-\Y\trans\A\Y\left(\Y\trans\bSig\Y-1\right)+\frac{\Y\trans\A\Y\left(\Y\trans\bSig\Y-1\right)^2}{\Y\trans\bSig\Y}+\Y\trans\A\Y-\frac{1}{p}\tr\A\right)\\
        &\times\left(-\Y\trans\B\Y\left(\Y\trans\bSig\Y-1\right)+\frac{\Y\trans\B\Y\left(\Y\trans\bSig\Y-1\right)^2}{\Y\trans\bSig\Y}+\Y\trans\B\Y-\frac{1}{p}\tr\B\right)\\
        =&\E \Y\trans\A\Y\Y\trans\B\Y\left(\Y\trans\bSig\Y-1\right)^2-\E \Y\trans\A\Y\left(\Y\trans\bSig\Y-1\right)\left(\Y\trans\B\Y-\frac{1}{p}\tr\B\right)\\
        &-\E \Y\trans\B\Y\left(\Y\trans\bSig\Y-1\right)\left(\Y\trans\A\Y-\frac{1}{p}\tr\A\right)\\
        &+\E \left(\Y\trans\A\Y-\frac{1}{p}\tr\A\right)\left(\Y\trans\B\Y-\frac{1}{p}\tr\B\right)+o(p^{-1})\frac{\|\A\|\B\|\|\bSig\|^3}{\lambda_{\min}^2(\bSig)}\\
        =&\frac{\tr(\A) \tr(\B)}{p^2}\E \left(\Y\trans\bSig\Y-1\right)^2-\frac{\tr(\A)}{p} \E \left(\Y\trans\bSig\Y-1\right)\left(\Y\trans\B\Y-\frac{1}{p}\tr\B\right)\\
        &-\frac{\tr(\B)}{p} \E \left(\Y\trans\bSig\Y-1\right)\left(\Y\trans\A\Y-\frac{1}{p}\tr\A\right)\\
        &+\E \left(\Y\trans\A\Y-\frac{1}{p}\tr\A\right)\left(\Y\trans\B\Y-\frac{1}{p}\tr\B\right)+o(p^{-1})\frac{\|\A\|\B\|\|\bSig\|^3}{\lambda_{\min}^2(\bSig)}\\
        =&\frac{\tau_3}{p^2}\left(\frac{\tr(\A) \tr(\B)\tr(\bSig \circ \bSig)}{p^2}-\frac{\tr(\A)\tr(\bSig \circ \B)}{p}-\frac{\tr(\B)\tr(\bSig \circ \A)}{p}+\tr(\A \circ \B)   \right)\\
        &+\frac{2}{p^2}\left( \frac{\tr(\A) \tr(\B)}{p^2}\tr(\bSig^2)-\frac{\tr(\A)}{p}\tr(\bSig \B)-\frac{\tr(\B)}{p}\tr(\bSig \A)+\tr(\A\B) \right)\\
        &+o(p^{-1})\frac{\|\A\|\B\|\|\bSig\|^3}{\lambda_{\min}^2(\bSig)}.
     \end{align*}
 
     For concentration inequality, by \eqref{form:identity_1} and Proposition \ref{prop:alpha>4_moments_of_self-normalized_quadratic_forms},
    \begin{align*}
        \E\left|\frac{\Y\trans\A\Y}{\Y\trans\bSig\Y}-\frac{\tr\A}{p}\right|^k=&\E\left|\frac{\Y\trans\A\Y}{\Y\trans\bSig\Y}-\Y\trans\A\Y+\Y\trans\A\Y-\frac{1}{p}\tr\A\right|^k\\
        \lesssim&\E\left|\frac{\Y\trans\A\Y\left(\Y\trans\bSig\Y-1\right)}{\Y\trans\bSig\Y}\right|^k+\E\left|\Y\trans\A\Y-\frac{1}{p}\tr\A\right|^k\\
        =&\begin{cases}
            O(p^{-1})\frac{\|\A\|^k\|\bSig\|^k}{\lambda_{\min}^{k}(\bSig)},&k=2,\\
            o(p^{-1})\frac{\|\A\|^k\|\bSig\|^k}{\lambda_{\min}^{k}(\bSig)},&k\geq 3.
        \end{cases}
    \end{align*}
\end{proof}

\subsection{Proof of Proposition \ref{prop:cross_term_self_normalization}}
\begin{proof}[Proof of Proposition \ref{prop:cross_term_self_normalization}]
    For $\alpha\geq 2$, the case of $k=2$ has been studied in the proof of Proposition \ref{prop:moments_of_self-normalized_quadratic_forms}, so we need to prove the result for $k\geq 4$. Since all the terms can be represented as the following summation 
    \begin{align}
        \sum_{i_1,\cdots,i_r}^* a_{i_{j_1}i_{j_2}}\cdots a_{i_{j_{2k-1}}i_{j_{2k}}}\E Y_{i_1}^{k_1}\cdots Y_{i_r}^{k_r},\label{form:cross_item}
    \end{align}
    where $k_1+\cdots+k_r=2k$, $j_1,\cdots,j_{2k}\in\{1,\cdots,r\}$ and $j_{2l-1}\not=j_{2l}$ for $l=1,\cdots,k$, so we only need to analyze the convergence order of \eqref{form:cross_item}. When $r\leq 3$, by Proposition \ref{lem:even_number} and Proposition \ref{lem:odd_number_2}, we have 
    \begin{align*}
        \E Y_{i_1}^{k_1}\cdots Y_{i_r}^{k_r}=o(p^{-r}).
    \end{align*}
    And by Lemma \ref{lem:summation_of_sigma}, 
    \begin{align*}
        \left|\sum_{i_1,\cdots,i_r}^* a_{i_{j_1}i_{j_2}}\cdots a_{i_{j_{2k-1}}i_{j_{2k}}}\right|=O(p)\|\A\|^k,
    \end{align*}
    which leads to $\eqref{form:cross_item}=o(p^{-1})\|\A\|^k$.
    When $r\geq 4$, by Proposition \ref{lem:even_number} and Proposition \ref{lem:odd_number_2},
    \begin{align*}
        \E Y_{i_1}^{k_1}\cdots Y_{i_r}^{k_r}=O(p^{-r}).
    \end{align*}
    And by Lemma \ref{lem:summation_of_sigma},
    \begin{align*}
        \left|\sum_{i_1,\cdots,i_r}^* a_{i_{j_1}i_{j_2}}\cdots a_{i_{j_{2k-1}}i_{j_{2k}}}\right|=O(p^{\frac{r+\lfloor\frac{r}{4}\rfloor}{2}})\|\A\|^k.
    \end{align*}
    Therefore, we have $\eqref{form:cross_item}=o(p^{-3/2})\|\A\|^k$. To sum up, we conclude that 
    \begin{align*}
        \E \left| \sum_{i \neq j} a_{ij} Y_i Y_j\right|^k=\begin{cases}
            O(p^{-1})\|\A\|^k,&k=2,\\
            o(p^{-1})\|\A\|^k,&k\geq 3.
        \end{cases}
        \end{align*}

    For $\alpha>4$, the case of $k\leq 4$ has been analyzed in the proof of Proposition \ref{prop:alpha>4_moments_of_self-normalized_quadratic_forms}, so we need to study the item \eqref{form:cross_item} for $k\geq 6$. To make it more clear, we consider the following cases:
    \begin{itemize}
        \item Case 1:  When $r\leq 3$, by Lemma \ref{lem:summation_of_sigma}, 
            \begin{align*}
                \left|\sum_{i_1,\cdots,i_r}^*a_{i_{j_1}i_{j_2}}\cdots a_{i_{j_{2k-1}}i_{j_{2k}}}\right|=O(p)\|\A\|^k.
            \end{align*}
            And by Proposition \ref{prop:alpha>4_general_number}, we have 
            \begin{align*}
                \E Y_{i_1}^{k_1}\cdots Y_{i_r}^{k_r}=o(p^{-4})
            \end{align*}
            which leads to $\eqref{form:alpha>4_item_2}=o(p^{-3}).$
            \item Case 2:  When $4\leq r\leq5$, by Proposition \ref{prop:alpha>4_general_number},
            \begin{align*}
                \E Y_1^{k_1}\cdots Y_r^{k_r}=\begin{cases}
                    O(p^{-6}),&k=6,\\
                    o(p^{-6}),&k\geq8.
                \end{cases}
            \end{align*}
            And by Lemma \ref{lem:summation_of_sigma},
            \begin{align*}
                \left|\sum_{i_1,\cdots,i_r}^*a_{i_{j_1}i_{j_2}}\cdots a_{i_{j_{2k-1}}i_{j_{2k}}}\right|=O(p^{3})\|\A\|^k.
            \end{align*}
            So we have 
            \begin{align*}
                \eqref{form:cross_item}=\begin{cases}
                    O(p^{-3})\|\A\|^k,&k=6,\\
                    o(p^{-3})\|\A\|^k,&k\geq 8.
                \end{cases}
            \end{align*}
            \item Case 3:  When $6\leq r\leq k-1$ and $k\geq 8$, by Proposition \ref{prop:alpha>4_general_number},
            \begin{align*}
                \E Y_{i_1}^{k_1}\cdots Y_{i_r}^{k_r}=O(p^{-r-1}).
            \end{align*}
            And by Lemma \ref{lem:summation_of_sigma},
            \begin{align*}
                \left|\sum_{i_1,\cdots,i_r}^*a_{i_{j_1}i_{j_2}}\cdots a_{i_{j_{2k-1}}i_{j_{2k}}}\right|=O(p^{\frac{r+\lfloor\frac{r}{4}\rfloor}{2}})\|\A\|^k.
            \end{align*}
            We have $\eqref{form:alpha>4_item_2}=O(p^{-7/2})\|\A\|^k$.
            \item Case 4:  When $r=k$, If at least one of $k_1,\cdots,k_r$ is odd, by Proposition \ref{prop:alpha>4_general_number}, 
            \begin{align*}
                \E Y_{i_1}^{k_1}\cdots Y_{i_r}^{k_r}=O(p^{-r-1}).
            \end{align*} 
            And by Lemma \ref{lem:summation_of_sigma},
            \begin{align*}
                \left|\sum_{i_1,\cdots,i_r}^*a_{i_{j_1}i_{j_2}}\cdots a_{i_{j_{2k-1}}i_{j_{2k}}}\right|=O(p^{\frac{r+\lfloor\frac{r}{4}\rfloor}{2}})\|\A\|^k.
            \end{align*}
            Otherwise, $k_1=\cdots=k_r=2$ and in this situation, by Proposition \ref{prop:alpha>4_general_number},
            \begin{align*}
                \E Y_{i_1}^{2}\cdots Y_{i_r}^{2}=O(p^{-r}).
            \end{align*}
            And we must have 
            \begin{align*}
                \left|\sum_{i_1,\cdots,i_6}^*a_{i_{j_1}i_{j_2}}\cdots a_{i_{j_{11}}i_{j_{12}}}\right|=O(p^{\frac{r}{2}})\|\A\|^k.
            \end{align*}
            Both two cases lead to 
            \begin{align*}
            \eqref{form:alpha>4_item_2}=\begin{cases}
                O(p^{-3})\|\A\|^k,&k=6,\\
                O(p^{-4})\|\A\|^k,&k\geq 8.
            \end{cases}
            \end{align*}
            \item Case 5:  When $k+1\leq r\leq 2k$, by Proposition \ref{prop:alpha>4_general_number},
            \begin{align*}
                \E Y_{i_1}^{2}\cdots Y_{i_r}^{2}=O(p^{-r}).
            \end{align*}
            And by Lemma \ref{lem:summation_of_sigma},
            \begin{align*}
                \left|\sum_{i_1,\cdots,i_r}^*a_{i_{j_1}i_{j_2}}\cdots a_{i_{j_{2k-1}}i_{j_{2k}}}\right|=O(p^{\frac{r+\lfloor\frac{r}{4}\rfloor}{2}})\|\A\|^k.
            \end{align*}
            Therefore, we have $\eqref{form:cross_item}=O(p^{-4})$.
    \end{itemize}
    Collecting all the above cases, we conclude that 
    \begin{align*}
        \E \left| \sum_{i \neq j} a_{ij} Y_i Y_j\right|^k=\begin{cases}
            O(p^{-\frac{k}{2}})\|\A\|^k,&2\leq k\leq 6,\\
            o(p^{-3})\|\A\|^k,&k\geq 7.
        \end{cases}
        \end{align*}

\end{proof}

\section{Proofs of  LSD}
\subsection{Proof of Theorem \ref{thm:lsd1}}
\begin{proof}
We study the diagonal and off-diagonal elements separately. In special, we have
\begin{align*}
    \frac{1}{p}\left\|p\cdot\E\frac{\Y\Y\trans}{\Y\trans\bSig\Y}-\bI_p\right\|_F^2=\frac{1}{p}\sum_{i=1}^p\left(p\cdot\E\frac{Y_i^2}{\Y\trans\bSig\Y}-1\right)^2+\frac{1}{p}\sum_{i\not=j}\left(p\cdot\E\frac{Y_iY_j}{\Y\trans\bSig\Y}\right)^2.
\end{align*}
It remains to show
    \begin{align}
        &p\cdot\E\frac{Y_i^2}{\Y\trans\bSig\Y}-1\to 0,\quad\text{uniformly in $1\leq i\leq p$,}\label{form:lim1}\\
        &p^{\frac{3}{2}}\cdot\left|\E\frac{Y_iY_j}{\Y\trans\bSig\Y}\right|\to 0,\quad\text{uniformly in $1\leq i,j\leq p$.}\label{form:lim2}
    \end{align}    
\begin{itemize}
    \item \textbf{Diagonal elements:}  As for \eqref{form:lim1}, we have
    \begin{align*}
        \E\frac{Y_i^2}{\Y\trans\bSig\Y}=2\E Y_i^2-\E Y_i^2\Y\trans\bSig\Y+\E\frac{Y_i^2\left(\Y\trans\bSig\Y-1\right)^2}{\Y\trans\bSig\Y}:=\frac{2}{p}-\varepsilon_1+\varepsilon_2.
    \end{align*}
By Lemmas \ref{lem:even_number} and \ref{lem:odd_number_1}, 
    \begin{align*}
        \varepsilon_1=&\sum_{j,k}\sigma_{jk}\E Y_i^2Y_jY_k=\sigma_{ii}\E Y_i^4+\sum_{j\not=i}\sigma_{jj}\E Y_i^2Y_j^2+\sum_{j,k\not=i}^*\sigma_{jk}\E Y_i^2Y_jY_k\\
        =&\frac{1}{p}+o(\frac{1}{p}).
    \end{align*}
Moreover
    \begin{align*}
        \varepsilon_2 \leq &\frac{1}{\lambda_{min}(\bSig)}\E Y_i^2\left(\Y\trans\bSig\Y-1\right)^2\\
        \leq &\frac{2}{\lambda_{min}(\bSig)}\E Y_i^2\left(\sum_{j}(\sigma_{jj}-1)Y_j^2\right)^2+\frac{2}{\lambda_{min}(\bSig)}\E Y_i^2\left(\sum_{j,k}^*\sigma_{jk}Y_jY_k\right)^2.
    \end{align*}
    For the first term, by Lemma \ref{lem:even_number} 
    \begin{align*}
        &\E Y_i^2\left(\sum_{j}(\sigma_{jj}-1)Y_j^2\right)^2\\
        =&\E Y_i^2\left(\sum_{j}(\sigma_{jj}-1)^2Y_j^4+\sum_{j,k}^*(\sigma_{jj}-1)(\sigma_{kk}-1)Y_j^2Y_k^2\right)\\
        =&(\sigma_{ii}-1)^2\E Y_i^6+\sum_{j\not=i}(\sigma_{jj}-1)\E Y_i^2Y_j^2\\
        &+2\sum_{j\not=i}(\sigma_{ii}-1)(\sigma_{jj}-1)\E Y_i^4Y_j^2\\
        &+\sum_{j,k\not=i}^*(\sigma_{jj}-1)(\sigma_{kk}-1)\E Y_i^2Y_j^2Y_k^2\\
        =&o(p^{-1}).
    \end{align*}
    For the second term, by Holder's inequality, Lemma \ref{lem:even_number} and Proposition \ref{prop:cross_term_self_normalization},
    \begin{align}
        \E Y_i^2\left(\sum_{j,k}^*\sigma_{jk}Y_jY_k\right)^2\leq\sqrt{\E Y_1^4}\sqrt{\E\left(\sum_{i,j}^*\sigma_{ij}Y_iY_j\right)^4}=o(p^{-1}).\label{form:cross_fourth_moment}
    \end{align}
Combine all the pieces, uniformly we have
\begin{align*}
    p\cdot\E\frac{Y_i^2}{\Y\trans\bSig\Y}-1\to 0.
\end{align*}
    \item \textbf{Off-diagonal elements:}  As for \eqref{form:lim2}, 
    we consider the decomposition 
    \begin{align*}
        \E\frac{Z_iZ_j}{\Z\trans\bSig\Z}=&\E\frac{Z_iZ_j}{\sum_{k=1}^p\sigma_{kk}Z_k^2}-\E\frac{Z_iZ_j\sum_{k,l}^*\sigma_{kl}Z_kZ_l}{\left(\sum_{k=1}^p\sigma_{kk}Z_k^2\right)^2}\\
        &+\E\frac{Z_iZ_j\left(\sum_{k,l}^*\sigma_{kl}Z_kZ_l\right)^2}{\left(\sum_{k=1}^p\sigma_{kk}Z_k^2\right)^2\Z\trans\bSig\Z}.
    \end{align*}
    To deal with the first two items, we need Lemma \ref{lem:non_identity_diagonal_self_normalization} which extend Proposition \ref{lem:odd_number_2}. By Lemma \ref{lem:non_identity_diagonal_self_normalization}, 
    \begin{align*}
        &\left|\E\frac{Z_1Z_2}{\sum_{k=1}^p\sigma_{kk}Z_k^2}\right|=o(p^{-2+\delta}),\\
        &\left|\E\frac{Z_1^2Z_2Z_3}{\left(\sum_{k=1}^p\sigma_{kk}Z_k^2\right)^2}\right|=o(p^{-3+\delta}),\\
        &\left|\E\frac{Z_1Z_2Z_3Z_4}{\left(\sum_{k=1}^p\sigma_{kk}Z_k^2\right)^2}\right|=o(p^{-4+\delta}),
    \end{align*}
    which yields
    \begin{align*}
        \left|\E\frac{Z_iZ_j\sum_{k,l}^*\sigma_{kl}Z_kZ_l}{\left(\sum_{k=1}^p\sigma_{kk}Z_k^2\right)^2}\right|=o(p^{-2+\delta}).
    \end{align*}
For the third term, by Holder's inequality, Lemma \ref{lem:even_number} and Proposition \ref{prop:cross_term_self_normalization}, 
    \begin{align*}
        &\left|\E\frac{Z_iZ_j\left(\sum_{k,l}^*\sigma_{kl}Z_kZ_l\right)^2}{\left(\sigma_{11}Z_1^2+\cdots+\sigma_{pp}Z_p^2\right)^2\Z\trans\bSig\Z}\right|=\left|\E\frac{Y_iY_j\left(\sum_{k,l}^*\sigma_{kl}Y_kY_l\right)^2}{\left(\Y \trans \diag(\bSig) \Y \right)^2\Y \trans\bSig\Y }\right|\\
        \leq & \frac{1}{\lambda_{min}(\bSig)^3}\sqrt{\E Y_i^2Y_j^2}\sqrt{\E\left(\sum_{i,j}^*\sigma_{ij}Y_iY_j\right)^4}=o(p^{-\frac{3}{2}}).
    \end{align*}
    Combine all the pieces, uniformly we have
    \begin{align*}
        p^{\frac{3}{2}}\cdot\left|\E\frac{Y_iY_j}{\Y\trans\bSig\Y}\right|\to 0.
    \end{align*}
\end{itemize}
The proof is completed.
\end{proof}

\section{Proofs of CLT}
\subsection{Proof of Theorem \ref{thm:clt1}}
\begin{proof}
It remains to show that 
    \begin{align}
        &p^2\E\frac{Y_i^2}{\Y\trans\bSig\Y}-p+(\tau-1)\sigma_{ii}-\frac{2\tr(\bSig^2)+(\tau-3)\tr(\bSig\circ\bSig)}{p}\to 0,\label{form:clt_item_1}\\
        &p^{\frac{5}{2}}\E\frac{Y_iY_j}{\Y\trans\bSig\Y}+2p^{\frac{1}{2}}\sigma_{ij}\to 0.\label{form:clt_item_2}
    \end{align}

    As for \eqref{form:clt_item_1}, we consider
    \begin{align*}
        \E\frac{Y_i^2}{\Y\trans\bSig\Y}=&\E Y_i^2+\E Y_i^2(1-\Y\trans\bSig\Y)+\E Y_i^2(1-\Y\trans\bSig\Y)^2\\
        &+\E Y_i^2(1-\Y\trans\bSig\Y)^3+\E\frac{Y_i^2(1-\Y\trans\bSig\Y)^4}{\Y\trans\bSig\Y}.
    \end{align*}
    First, by Proposition \ref{prop:even_number4} and Proposition \ref{prop:alpha>4_general_number}, we have
    \begin{align*}
        &\E Y_i^2\Y\trans\bSig\Y\\
        =&\sum_{j,k}\sigma_{jk}\E Y_i^2Y_jY_k =\sigma_{ii}\E Y_i^4+\sum_{j\not=i}\sigma_{jj}\E Y_i^2Y_j^2+\sum_{j,k\not=i}^*\sigma_{jk}\E Y_i^2Y_jY_k\\
            =&\sigma_{ii}\E Y_1^4+(p-\sigma_{ii})\frac{1-p\E Y_1^4}{p(p-1)}+o(p^{-3})=\frac{1}{p}+\frac{\tau-1}{p^2}\sigma_{ii}-\frac{\tau-1}{p^2}+o(p^{-2}).
    \end{align*}
    Next, we consider
    \begin{align*}
        &\E Y_i^2\left(1-\Y\trans\bSig\Y\right)^2\\
        =&\E Y_i^2\left(\sum_{j}(\sigma_{jj}-1)Y_j^2+\sum_{k,l}^*\sigma_{kl}Y_kY_l\right)^2\\
        =&\E Y_i^2\left(\sum_{j}(\sigma_{jj}-1)Y_j^2\right)^2+2\E Y_i^2\left(\sum_{j}(\sigma_{jj}-1)Y_j^2\right)\left(\sum_{k,l}^*\sigma_{kl}Y_kY_l\right)\\
        &+\E Y_i^2\left(\sum_{k,l}^*\sigma_{kl}Y_kY_l\right)^2.
    \end{align*}
    For the first term,
    \begin{align*}
        &\E Y_i^2\left(\sum_{j}(\sigma_{jj}-1)Y_j^2\right)^2\\
        =&\sum_{j}(\sigma_{jj}-1)^2\E Y_i^2Y_j^4+\sum_{j,k}^*(\sigma_{jj}-1)(\sigma_{kk}-1)\E Y_i^2Y_j^2Y_k^2\\
        =&(\sigma_{ii}-1)\E Y_1^6+\sum_{j\not=i}(\sigma_{jj}-1)^2\E Y_1^2Y_2^4+2\sum_{j\not=i}(\sigma_{ii}-1)(\sigma_{jj}-1)\E Y_1^4Y_2^2\\
        &+\sum_{j,k\not=i}^*(\sigma_{jj}-1)(\sigma_{kk}-1)\E Y_1^2Y_2^2Y_3^2+o(p^{-2}).
    \end{align*}
    Since that
    \begin{align*}
        \sum_{j}(\sigma_{jj}-1)^2=&\tr(\bSig\circ\bSig)-p,\\
        \sum_{j,k}^*(\sigma_{jj}-1)(\sigma_{kk}-1)=&-\sum_{j}(\sigma_{jj}-1)^2=-\tr(\bSig\circ\bSig)+p,
    \end{align*}
    by Proposition \ref{prop:even_number4}, we have 
    \begin{align*}
        \E Y_i^2\left(\sum_{j}(\sigma_{jj}-1)Y_j^2\right)^2=&\frac{\left(\tr(\bSig\circ\bSig)-p\right)(\tau-1)}{p^3}\\
        =&\frac{\tr(\bSig\circ\bSig)(\tau-1)}{p^3}-\frac{\tau-1}{p^2}+o(p^{-2}).
    \end{align*}
    For the second term, by Proposition \ref{prop:alpha>4_general_number},
    \begin{align*}
        &\E Y_i^2\left(\sum_{j}(\sigma_{jj}-1)Y_j^2\right)\left(\sum_{k,l}^*\sigma_{kl}Y_kY_l\right)\\
        =&\sum_{j}\sum_{k,l}^*(\sigma_{jj}-1)\sigma_{kl}\E Y_i^2Y_j^2Y_kY_l\\
        =&\sum_{j,k}^*(\sigma_{jj}-1)\sigma_{jk}\E Y_i^2Y_j^3Y_k+\sum_{j,k,l}^*(\sigma_{jj}-1)\sigma_{kl}\E Y_i^2Y_j^2Y_kY_l\\
        =&O(p^{-\frac{5}{2}}).
    \end{align*}
    For the third term, by Proposition \ref{prop:even_number4} and Proposition \ref{prop:alpha>4_general_number},
    \begin{align*}
        &\E Y_i^2\left(\sum_{k,l}^*\sigma_{kl}Y_kY_l\right)^2\\
        =&\sum_{k\not=l,s\not=t}\sigma_{kl}\sigma_{st}\E Y_i^2Y_kY_lY_sY_t\\
        =&2\sum_{k,l}^*\sigma_{kl}^2\E Y_i^2Y_k^2Y_l^2+4\sum_{k,l,s}^*\sigma_{kl}\sigma_{ks}\E Y_i^2Y_k^2Y_lY_k+\sum_{k,l,s,t}^*\sigma_{kl}\sigma_{st}\E Y_i^2Y_kY_lY_sY_t\\
        =&4\sum_{k\not=i}\sigma_{ik}^2\E Y_1^4Y_2^2+2\sum_{k,l\not=i}^*\sigma_{kl}^2\E Y_1^2Y_2^2Y_3^2+O(p^{-3})\\
        =&2\frac{\tr(\bSig^2)-\tr(\bSig\circ\bSig)}{p^3}+o(p^{-2}).
    \end{align*}
    So we have concluded that 
    \begin{align}
        \E Y_i^2(1-\Y\trans\bSig\Y)^2=\frac{2\tr(\bSig^2)}{p^3}+\frac{\tr(\bSig\circ\bSig)(\tau-3)}{p^3}-\frac{\tau-1}{p^2}+o(p^{-2}).
    \end{align}
    Then, we consider 
    \begin{align}
        &\E\frac{Y_i^2(1-\Y\trans\bSig\Y)^4}{\Y\trans\bSig\Y}\lesssim  \E Y_i^2(1-\Y\trans\bSig\Y)^4 \nonumber \\
        \lesssim &\E Y_i^2\left(\sum_{j}(\sigma_{jj}-1)Y_j^2\right)^4+\E Y_i^2\left(\sum_{j,k}^*\sigma_{jk}Y_jY_k\right)^4.\label{form:I_5_control}
    \end{align}
    For the first term, we expand the fourth power and it consists of the summation like 
    \begin{align}
        \sum_{i_1,\cdots,i_r}^*(\sigma_{i_1i_1}-1)^{k_1}\cdots(\sigma_{i_ri_r}-1)^{k_r}\E Y_i^2Y_{i_1}^{2k_1}\cdots Y_{i_r}^{2k_r},\label{form:Y_i^2_sigma_summation}
    \end{align}
    where $k_1+\cdots+k_r=4$. On the one hand, if all the indexes $\{i_1,\cdots,i_r\}$ run through $\{1,2,\cdots,p\}\backslash\{i\}$, then by Proposition \ref{prop:alpha>4_general_number},
    \begin{align}
        \E Y_i^2Y_{i_1}^{k_1}\cdots Y_{i_r}^{k_r}=O(p^{-1-|I_1|-2|I_2|-\frac{\beta}{2}|I_3|}).\label{form:Y_i^2}
    \end{align}
    Combined with \eqref{form:1-power_numbers} and \eqref{form:Y_i^2}, we have 
    \begin{align*}
        \eqref{form:Y_i^2_sigma_summation}=O(p^{r-\lceil\frac{|I_1|}{2}\rceil-1-|I_1|-2|I_2|-\frac{\beta}{2}|I_3|}).
    \end{align*}
    Since that $r=|I_1|+|I_2|+|I_3|$, the order of $p$ is 
    \begin{align}
        & r-\lceil\frac{|I_1|}{2}\rceil-1-|I_1|-2|I_2|-\frac{\beta}{2}|I_3| \nonumber \\
        =&-\lceil\frac{|I_1|}{2}\rceil-1-|I_2|-(\frac{\beta}{2}-1)|I_3|<-2.\label{form:result_case1}
    \end{align}
    On the other hand, if one of the indexes equals to $i$ and others run through $\{1,2,\cdots,p\}\backslash\{i\}$, it suffices to consider 
    \begin{align}
        (\sigma_{ii}-1)^{k_r}\sum_{i_1,\cdots,i_{r-1}}^*(\sigma_{i_1i_1}-1)^{k_1}\cdots(\sigma_{i_{r-1}i_{r-1}}-1)^{k_{r-1}}\E Y_i^{2+2k_r}Y_{i_1}^{2k_1}\cdots Y_{i_r}^{2k_{r-1}},\label{form:Y_i^2_sigma_summation_r-1}
    \end{align}
    where $k_1+\cdots+k_r=4$. We denote $I_1=\{1\leq j\leq r-1:k_j=1\}$, $I_2=\{1\leq j\leq r-1:k_j=2\}$, $I_3=\{1\leq j\leq r-1:k_j\geq 3\}$,
    and by Proposition \ref{prop:even_number4}, we have 
    \begin{align}
        \E Y_i^{2+2k_r}Y_{i_1}^{2k_1}\cdots Y_{i_r}^{2k_{r-1}}=O(p^{-|I_1|-2|I_2|-\frac{\beta}{2}|I_3|-2\one(k_r=1)-\frac{\beta}{2}\one(k_r\geq 2)}).\label{form:Y_i^2_r-1}
    \end{align}
    Combined with \eqref{form:1-power_numbers} and \eqref{form:Y_i^2_r-1}, we have 
    \begin{align}
        \eqref{form:Y_i^2_sigma_summation_r-1}=O(p^{r-1-\lceil\frac{|I_1|}{2}\rceil-|I_1|-2|I_2|-\frac{\beta}{2}|I_3|-2\one(k_r=1)-\frac{\beta}{2}\one(k_r\geq 2)}).
    \end{align}
    Since that $r-1=|I_1|+|I_2|+|I_3|$, the order of $p$ is 
    \begin{align}
        -\lceil\frac{|I_1|}{2}\rceil-|I_2|-(\frac{\beta}{2}-1)|I_3|-2\one(k_r=1)-\frac{\beta}{2}\one(k_r\geq 2)<-2.\label{form:result_case2}
    \end{align}
    Together with \eqref{form:result_case1} and \eqref{form:result_case2}, we have $\eqref{form:Y_i^2_sigma_summation}=o(p^{-2})$.
    For the second term, by Holder's inequality, Proposition \ref{prop:even_number4} and Proposition \ref{prop:cross_term_self_normalization}, we have
    \begin{align*}
        \E Y_i^2\left(\sum_{j,k}^*\sigma_{jk}Y_jY_k\right)^4\leq\sqrt{\E Y_i^4}\sqrt{\left(\sum_{j,k}^*\sigma_{jk}Y_jY_k\right)^8}=o(p^{-\frac{5}{2}}).
    \end{align*}
    So the right side of \eqref{form:I_5_control} can be bounded in $o(p^{-2})$ and we obtain that $|I_5|=o(p^{-2})$.
    Finally, $I_4$ can be controlled by $I_3$ and $I_5$ with the assistance of Holder's inequality,
    \begin{align*}
        \left|\E Y_i^2(1-\Y\trans\bSig\Y)^3\right|\leq\sqrt{\E Y_i^2(1-\Y\trans\bSig\Y)^2}\sqrt{\E Y_i^2(1-\Y\trans\bSig\Y)^4}=o(p^{-2}).
    \end{align*}
    To sum up, we have 
    \begin{align*}
        p\cdot\E\frac{Y_i^2}{\Y\trans\bSig\Y}=1-\frac{\tau-1}{p}\sigma_{ii}+\frac{2\tr(\bSig^2)+(\tau-3)\tr(\bSig\circ\bSig)}{p^2}+o(p^{-1}).
    \end{align*}

    As for \eqref{form:clt_item_2}, 
    we decompose the quadratic form into diagonal part and non-diagonal part,
    \begin{align*}
        &\E\frac{Z_iZ_j}{\Z\trans\bSig\Z}\\
        =&\E\frac{Z_iZ_j}{\sum_{k=1}^p\sigma_{kk}Z_k^2}-\E\frac{Z_iZ_j\left(\sum_{k\not=l}\sigma_{kl}Z_kZ_l\right)}{\left(\sum_{k=1}^p\sigma_{kk}Z_k^2\right)^2}+\E\frac{Z_iZ_j\left(\sum_{k\not=l}\sigma_{kl}Z_kZ_l\right)^2}{\left(\sum_{k=1}^p\sigma_{kk}Z_k^2\right)^3}\\
        &-\E\frac{Z_iZ_j\left(\sum_{k\not=l}\sigma_{kl}Z_kZ_l\right)^3}{\left(\sum_{k=1}^p\sigma_{kk}Z_k^2\right)^4}+\E\frac{Z_iZ_j\left(\sum_{k\not=l}\sigma_{kl}Z_kZ_l\right)^4}{\left(\sum_{k=1}^p\sigma_{kk}Z_k^2\right)^4\Z\trans\bSig\Z}\\
        :=&T_1-T_2+T_3-T_4+T_5.
    \end{align*}
    As what we have done in the proof of LSD, here we need to extend Proposition \ref{prop:alpha>4_general_number} to a more general form in Lemma \ref{lem:alpha_4_diagonal}.
    By Lemma \ref{lem:alpha_4_diagonal}, we have $|T_1|=O(p^{-3})$. As for $T_2$,
    \begin{align*}
        T_2=&2\sigma_{ij}\E\frac{Z_i^2Z_j^2}{\left(\sum_{k=1}^p\sigma_{kk}Z_k^2\right)^2}+4\sum_{k\not=i,j}\sigma_{ik}\E\frac{Z_i^2Z_jZ_k}{\left(\sum_{k=1}^p\sigma_{kk}Z_k^2\right)^2}\\&+\sum_{k,l\not=i,j}^*\sigma_{kl}\E\frac{Z_iZ_jZ_kZ_l}{\left(\sum_{k=1}^p\sigma_{kk}Z_k^2\right)^2}.
    \end{align*}
    The first term is the dominant term, that is 
    \begin{align*}
        &\E\frac{Z_i^2Z_j^2}{\left(\sum_{k=1}^p\sigma_{kk}Z_k^2\right)^2}
        =\E\frac{Y_i^2Y_j^2}{\left(\sum_{k=1}^p\sigma_{kk}Y_k^2\right)^2}\\
        =&\E Y_i^2Y_j^2\left[2-\left(\sum_{k=1}^p\sigma_{kk}Y_k^2\right)^2+\frac{\left(1-\left(\sum_{k=1}^p\sigma_{kk}Y_k^2\right)^2\right)^2}{\left(\sum_{k=1}^p\sigma_{kk}Y_k^2\right)^2}\right]\\
        =&\E Y_i^2Y_j^2\left[1+2\left(1-\sum_{k=1}^p\sigma_{kk}Y_k^2\right)-\left(1-\sum_{k=1}^p\sigma_{kk}Y_k^2\right)^2\right]\\
        &+\E Y_i^2Y_j^2 \frac{\left(1-\left(\sum_{k=1}^p\sigma_{kk}Y_k^2\right)^2\right)^2}{\left(\sum_{k=1}^p\sigma_{kk}Y_k^2\right)^2}.
    \end{align*}
    By Proposition \ref{prop:even_number4}, Proposition \ref{prop:alpha>4_moments_of_self-normalized_quadratic_forms} and Holder's inequality,
    \begin{align*}
        \E Y_i^2Y_j^2=&\frac{1-p\E Y_i^4}{p(p-1)}=\frac{1}{p^2}+O(p^{-3}),\\
        \E Y_i^2Y_j^2\left(\sum_{k=1}^p\sigma_{kk}Y_k^2-1\right)=&(\sigma_{ii}-1)\E Y_i^4Y_j^2+(\sigma_{jj}-1)Y_i^2Y_j^4\\
        &+\sum_{k\not=i,j}(\sigma_{kk}-1)\E Y_i^2Y_j^2Y_k^2=O(p^{-3}),\\
        \E Y_i^2Y_j^2\left(1-\sum_{k=1}^p\sigma_{kk}Y_k^2\right)^2\leq&\sqrt{\E Y_i^4Y_j^4}\sqrt{\E\left(1-\sum_{k=1}^p\sigma_{kk}Y_k^2\right)^4}=o(p^{-\frac{5}{2}}),\\
        \E Y_i^2Y_j^2\frac{\left(1-\left(\sum_{k=1}^p\sigma_{kk}Y_k^2\right)^2\right)^2}{\left(\sum_{k=1}^p\sigma_{kk}Y_k^2\right)^2}\leq&\E Y_i^2Y_j^2\left(1-\left(\sum_{k=1}^p\sigma_{kk}Y_k^2\right)^2\right)^2\\
        &\lesssim\E Y_i^2Y_j^2\left(1-\sum_{k=1}^p\sigma_{kk}Y_k^2\right)^2=o(p^{-\frac{5}{2}}),
    \end{align*}
    so we have 
    \begin{align*}
        \E\frac{Z_i^2Z_j^2}{\left(\sum_{k=1}^p\sigma_{kk}Z_k^2\right)^2}=\frac{1}{p^2}+o(p^{-\frac{5}{2}}).
    \end{align*}
    The second term and the third term can be controlled in $O(p^{-3})$ by Lemma \ref{lem:alpha_4_diagonal}, so we obtain that $|T_2|=-2\sigma_{ij}/p^2+o(p^{-\frac{5}{2}})$. Next for $T_3$,
    \begin{align*}
        T_3=&2\sum_{k,l}^*\sigma_{kl}^2\E\frac{Z_iZ_jZ_k^2Z_l^2}{\left(\sum_{k=1}^p\sigma_{kk}Z_k^2\right)^3}+4\sum_{k,l,s}^*\sigma_{kl}\sigma_{ks}\E\frac{Z_iZ_jZ_k^2Z_lZ_s}{\left(\sum_{k=1}^p\sigma_{kk}Z_k^2\right)^3}\\
        &+\sum_{k,l,s,t}^*\sigma_{kl}\sigma_{st}\E\frac{Z_iZ_jZ_kZ_lZ_sZ_t}{\left(\sum_{k=1}^p\sigma_{kk}Z_k^2\right)^3}\\
        =&4\sigma_{ij}^2\E\frac{Z_i^3Z_j^3}{\left(\sum_{k=1}^p\sigma_{kk}Z_k^2\right)^3}+8\sum_{k\not=i,j}\sigma_{ik}^2\E\frac{Z_i^3Z_jZ_k^2}{\left(\sum_{k=1}^p\sigma_{kk}Z_k^2\right)^3}\\
        &+2\sum_{k,l\not=i,j}^*\sigma_{kl}^2\E\frac{Z_iZ_jZ_k^2Z_l^2}{\left(\sum_{k=1}^p\sigma_{kk}Z_k^2\right)^3}+16\sigma_{ij}\sum_{k\not=i,j}\sigma_{ik}\E\frac{Z_i^3Z_j^2Z_k}{\left(\sum_{k=1}^p\sigma_{kk}Z_k^2\right)^3}\\
        &+8\sum_{k\not=i,j}\sigma_{ki}\sigma_{kj}\E\frac{Z_i^2Z_j^2Z_k^2}{\left(\sum_{k=1}^p\sigma_{kk}Z_k^2\right)^3}+8\sum_{k,l\not=i,j}^*\sigma_{ik}\sigma_{il}\E\frac{Z_i^3Z_jZ_kZ_l}{\left(\sum_{k=1}^p\sigma_{kk}Z_k^2\right)^3}\\
        &+16\sum_{k,l\not=i,j}^*\sigma_{ik}\sigma_{kl}\E\frac{Z_i^2Z_jZ_k^2Z_l}{\left(\sum_{k=1}^p\sigma_{kk}Z_k^2\right)^3}+4\sum_{k,l,s\not=i,j}^*\sigma_{kl}\sigma_{ks}\E\frac{Z_iZ_jZ_k^2Z_lZ_s}{\left(\sum_{k=1}^p\sigma_{kk}Z_k^2\right)^3}\\
        &+\sum_{k,l,s,t}^*\sigma_{kl}\sigma_{st}\E\frac{Z_iZ_jZ_kZ_lZ_sZ_t}{\left(\sum_{k=1}^p\sigma_{kk}Z_k^2\right)^3}.
    \end{align*}
    All these terms can be controlled in $O(p^{-3})$ by direct applications of Lemma \ref{lem:alpha_4_diagonal}, so $|T_3|=O(p^{-3})$. Then for $T_4$, all the terms can be represented as 
    \begin{align}
        \sum_{i_1,\cdots,i_r}^*\sigma_{i_{j_1}i_{j_2}}\sigma_{i_{j_3}i_{j_4}}\sigma_{i_{j_{5}}i_{j_{6}}}\E\frac{Z_iZ_jZ_{i_1}^{k_1}\cdots Z_{i_r}^{k_r}}{\left(\sum_{k=1}^p\sigma_{kk}Z_k^2\right)^4},\label{form:error_term}
    \end{align}
    where $k_1+\cdots+k_r=6$. When $r\leq 3$, by Lemma \ref{lem:summation_of_sigma} we have 
    \begin{align*}
        \sum_{i_1,\cdots,i_r}^*\sigma_{i_{j_1}i_{j_2}}\sigma_{i_{j_3}i_{j_4}}\sigma_{i_{j_{5}}i_{j_{6}}}=O(p),
    \end{align*}
    and by Lemma \ref{lem:alpha_4_diagonal},
    \begin{align*}
        \E\frac{Z_iZ_jZ_{i_1}^{k_1}\cdots Z_{i_r}^{k_r}}{\left(\sum_{k=1}^p\sigma_{kk}Z_k^2\right)^4}=O(p^{-4}).
    \end{align*}
    So \eqref{form:error_term} is controlled in $O(p^{-3})$. 
    When $r=4$, the summation can only be one of the following three types
    \begin{align}
        &\sum_{i_1,i_2,i_3,i_4}^*\sigma_{i_1i_2}\sigma_{i_1i_3}\sigma_{i_1i_4}=O(p),\label{form:r=4_case1}\\
        &\sum_{i_1,i_2,i_3,i_4}^*\sigma_{i_1i_2}\sigma_{i_1i_3}\sigma_{i_3i_4}=O(p),\label{form:r=4_case2}\\
        &\sum_{i_1,i_2,i_3,i_4}^*\sigma_{i_1i_2}^2\sigma_{i_3i_4}=O(p^2),\label{form:r=4_case3}
    \end{align}
    and by Lemma \ref{lem:alpha_4_diagonal},
    \begin{align*}
        \E\frac{Z_iZ_jZ_{i_1}^{k_1}\cdots Z_{i_r}^{k_r}}{\left(\sum_{k=1}^p\sigma_{kk}Z_k^2\right)^4}=O(p^{-4}).
    \end{align*}
    So we only need analyze the case of \eqref{form:r=4_case3} more carefully. We note that
    \begin{align*}
        \sum_{i_1,i_2,i_3,i_4}^*\sigma_{i_1i_2}^2\sigma_{i_3i_4}\E\frac{Z_iZ_jZ_{i_1}^{k_1}\cdots Z_{i_r}^{k_r}}{\left(\sum_{k=1}^p\sigma_{kk}Z_k^2\right)^4}=\sigma_{ij}\sum_{i_1,i_2\not=i,j}^*\sigma_{i_1i_2}^2\E\frac{Z_i^2Z_j^2Z_{i_1}^2Z_{i_2}^2}{\left(\sum_{k=1}^p\sigma_{kk}Z_k^2\right)^4}+R,
    \end{align*}
    where all the cases in $R$ have 
    \begin{align*}
        \E\frac{Z_iZ_jZ_{i_1}^{k_1}\cdots Z_{i_r}^{k_r}}{\left(\sum_{k=1}^p\sigma_{kk}Z_k^2\right)^4}=O(p^{-5}).
    \end{align*}
    So we conclude that $\eqref{form:error_term}$ can be bounded in $O(p^{-3})$ for $r=4$. When $r=5$, the summation can only be 
    \begin{align*}
        \sum_{i_1,\cdots,i_5}^*\sigma_{i_1i_2}\sigma_{i_1i_3}\sigma_{i_4i_5}=O(p^2),
    \end{align*}
    and by Lemma \ref{lem:alpha_4_diagonal},
    \begin{align*}
        \E\frac{Z_iZ_jZ_{i_1}^{2}Z_{i_2}Z_{i_3}Z_{i_4}Z_{i_5}}{\left(\sum_{k=1}^p\sigma_{kk}Z_k^2\right)^4}=O(p^{-6}).
    \end{align*}
    For $r=6$, 
    the summation can only be 
    \begin{align*}
        \sum_{i_1,\cdots,i_6}^*\sigma_{i_1i_2}\sigma_{i_3i_4}\sigma_{i_5i_6}=O(p^3),
    \end{align*}
    and by Lemma \ref{lem:alpha_4_diagonal},
    \begin{align*}
        \E\frac{Z_iZ_jZ_{i_1}Z_{i_2}Z_{i_3}Z_{i_4}Z_{i_5}Z_{i_6}}{\left(\sum_{k=1}^p\sigma_{kk}Z_k^2\right)^4}=O(p^{-8}).
    \end{align*}
    To sum up, we have $|T_4|=O(p^{-3})$.
    Finally for $T_5$, to remove the influence of $\Z\trans\bSig\Z$, we have to use Holder's inequality and by Proposition \ref{prop:even_number4} and Proposition \ref{prop:cross_term_self_normalization}, we have
    \begin{align*}
        \left|\E\frac{Z_iZ_j\left(\sum_{k\not=l}\sigma_{kl}Z_kZ_l\right)^4}{\left(\sum_{k=1}^p\sigma_{kk}Z_k^2\right)^4\Z\trans\bSig\Z}\right|\lesssim\sqrt{\E Y_i^2Y_j^2}\sqrt{\E\left(\sum_{i,j}^*\sigma_{ij}Y_iY_j\right)^8}=o(p^{-\frac{5}{2}}).
    \end{align*}
    In conclusion, we have 
    \begin{align*}
        p^{\frac{5}{2}}\E\frac{Y_iY_j}{\Y\trans\bSig\Y}+2p^{\frac{1}{2}}\sigma_{ij}\to 0,
    \end{align*}
    and the proof is complete.
\end{proof}
\end{appendix}

 \bibliography{ref}

\end{document}